\documentclass[12pt]{article}

\usepackage{graphicx} 

\usepackage{amsfonts}
\usepackage{amsmath}
\usepackage{amssymb}
\usepackage{amsthm}
\usepackage{enumitem}
\usepackage[margin=1in]{geometry}

\usepackage{tikz-cd}
\usepackage{arydshln}
\usepackage{tikz}

\usepackage{multirow}

\usepackage[colorlinks=true, linkcolor=violet, citecolor=teal]{hyperref}

\title{Deformations of geometric structures\\ on a subclass of LVM threefolds}
\author{Matthieu Madera}
\date{\today}

\newcommand{\N}{\mathbb N}
\newcommand{\Z}{\mathbb Z}

\newcommand{\C}{\mathbb C}
\newcommand{\PP}{\mathbb P}
\newcommand{\W}{\mathbb C^n\setminus\{0\}}

\newcounter{thm}[section]

\newtheorem{Theorem}[thm]{Theorem}
\newtheorem{Lemma}[thm]{Lemma}
\newtheorem{Proposition}[thm]{Proposition}

\newtheorem{Remark}[thm]{Remark}
\newtheorem{Definition}[thm]{Definition}
\newtheorem{Corollary}[thm]{Corollary}
\newtheorem{Example}[thm]{Example}

\newtheorem{TheoremLetter}{Theorem}
\newtheorem{TheoremIntro}{Theorem}

\DeclareMathOperator\coker{coker}
\DeclareMathOperator\im{im}
\DeclareMathOperator\rk{rk}
\DeclareMathOperator\Hom{Hom}
\DeclareMathOperator\Vect{Vect}
\DeclareMathOperator\Aut{Aut}

\begin{document}

\maketitle

\begin{abstract}
    We study deformations of geometric structures on some LVM manifolds of complex dimension $3$. More precisely, we study resonant structures, a particular type of $(G,X)$-structures, via the Ehresmann-Thurston principle, and their link with the deformation of the complex structure of the LVM manifold in the sense of Kodaira, Spencer and Kuranishi. We describe the Kuranishi family of these manifolds and construct a family containing all of them and complete at every point.
\end{abstract}

\tableofcontents

\newpage

\section{Introduction}\label{Intro}

The objective of this paper is to study the geometry of deformations of some examples of non-Kähler compact complex manifolds that carry holomorphic $(G,X)$-structures (see Section \ref{Subsection G,X structures}) and the interplay between the deformation space of these geometries and the deformation space of the complex structure. We focus on some complex threefolds that are characterized by their type $(2,6,4)$ in the class of non-Kähler LVM manifolds (see Section \ref{Section LVM manifolds}). We will study deformations of the complex structure of such a threefold $N$, that is obtained by a quotient of $\C^*\times \C^2\setminus\{0\}$ by an action of $\Z^2$ given by a group homomorphism $\rho_N:\Z^2\to \mathrm{Diag}(\C^3)^\times$ with values in the group of invertible linear diagonal transformations of $\C^3$. From this action, we introduce the concept of resonance (see Definition \ref{Definition Resonances for an LVM manifold}), together with a complex Lie group $G_N$ of resonant transformations of $\C^*\times \C^2\setminus\{0\}$. The first principal result concerns complex deformations. It describes the Kuranishi family of an LVM manifold $N$ of type $(2,6,4)$, exhibiting at the same time new non-Kähler compact complex manifolds.

\begin{TheoremIntro}\label{TheoremIntro Action and Kuranishi}
    (Theorems \ref{Theorem The action is fpf and proper} and \ref{Theorem Kuranishi family of a LVM manifold of type 264}) There exists an open neighborhood $U$ of $\rho_N$ in $\mathrm{Hom}(\Z^2,G_N)$ such that every homomorphism in $U$ furnishes a fixed point free, proper and cocompact action of $\Z^2$ on $\C^*\times \C^2\setminus\{0\}$. Moreover, the resulting germ of family of compact complex manifolds with base space $U$ is the Kuranishi family of $N$.
\end{TheoremIntro}

An LVM manifold of type $(2,6,4)$ carries naturally a holomorphic flat and torsion-free affine connection. Such a structure can be seen as an example of holomorphic $(G,X)$-structures (see Section \ref{Subsection G,X structures}). Given an LVM manifold $N$ of type $(2,6,4)$, the Lie group $G_N$ of resonant transformations of $V=\C^*\times \C^2\setminus\{0\}$ induces a new geometric model $(G_N,V)$. The affine structure of $N$ can be seen as a complete $(G_N,V)$-structure (see Definition \ref{Definition completeness and Kleinian for GX structures}) and Theorem \ref{TheoremIntro Action and Kuranishi} gives the description of some (smooth) $(G_N,V)$-structures close to the natural one. However, according to the so-called Ehresmann-Thurston principle (see \ref{Theorem Ehresmann Thurston principle}), it does not describe an open set of the deformation space of $(G_N,V)$-structures, since the fundamental group of $N$ is isomorphic with $\Z^3$. The second principal result concerns the deformations of the natural $(G_N,V)$-structure of $N$ and the relationship with its previously described Kuranishi space.

\begin{TheoremIntro}\label{TheoremIntro Geometrization}
    (Theorem \ref{Theorem Geometrization of LVM manifolds}) Any small deformation of the complex structure $N$ carries a family of compatible, i.e. holomorphic, $(G_N,V)$-structures that is parametrized by $3,4$ or $5$ complex numbers, depending on the resonances of $N$. Moreover, among these structures, there exists a unique one that is uniformizable (see Definition \ref{Definition completeness and Kleinian for GX structures}), and it has the property of being complete (see Definition \ref{Definition completeness and Kleinian for GX structures}). Moreover, the germ of the induced forgetful map 
    $$
    \Hom(\Z^3,G_N)\longrightarrow K_N,
    $$
    where $K_N$ is the Kuranishi space of $N$, is an analytic fibration with smooth fibers. Both germs of spaces and the forgetful map are smooth if and only if $N$ is non-resonant, and in that case, the map is a holomorphic submersion.
\end{TheoremIntro}

\subsection{Geometric structures}

In the context of complex manifolds, the use of geometric structures can play a crucial role. The most striking example is that of Riemann surfaces and projective structures. Projective structures play a fundamental role in the famous \textbf{Uniformization Theorem} of Riemann surfaces (we refer to \cite{HPSG2010}). Every Riemann surface can be written as a quotient of an open set of $\PP^1(\C)$ by a group of automorphisms acting freely and properly through projective automorphisms. This type of projective structure is called a uniformizable structure (see Definition \ref{Definition completeness and Kleinian for GX structures}). For a compact Riemann surface $X$ of genus $g\geq2$, sheaf cohomology, and more precisely, the Serre duality theorem and the Riemann-Roch theorem (see \cite{Gu66}, p.80 and p.98) are used to show that the space of holomorphic projective structures on $X$, i.e. compatible with the complex structure, is a complex manifold of dimension $3g-3$ and among these structures, there exists a unique one that is uniformizable. 

Other examples of compact complex manifolds that carry geometric structures are the compact complex tori. A compact complex torus of dimension $n\geq1$ is a quotient of $\C^n$ by the free and proper action of $\Z^{2n}$ generated by a lattice $\Gamma$ of $\C^n$ acting by translations. Translations are automorphisms of $\C^n$ that preserve its standard Hermitian metric. This metric thus descends to a flat Kähler metric on a torus. In fact, compact complex tori are the only compact Kähler manifolds that are holomorphically parallelizable. Although Kählerian, compact complex tori are generically not projective, in the sense that they cannot be embedded as closed submanifolds of a complex projective space. Translations also preserve the standard affine connection of $\C^n$ and, therefore, compact complex tori are examples of manifolds that carry an affine structure (see Example \ref{Example of G,X struct}).

It took the work of Hopf to discover the first example of compact complex non-Kählerian manifold. The example he studies in \cite{Ho48} is the quotient of $\C^2\setminus\{0\}$ by the action of $\Z$ generated by the contracting homothety $z\mapsto \frac{1}{2}z$. The topology of such a quotient is that of the product of spheres $\mathbb S^1\times \mathbb S^3$, for which the existence of a Kähler metric is obstructed by its topology. Since Hopf's work, a \textbf{Hopf manifold} is a complex manifold of dimension $n\geq 2$ whose universal covering is biholomorphic to $\W$ and whose fundamental group is isomorphic to $\Z$. It is shown independently in \cite{Bour2025} and \cite{Mad25} that every Hopf manifold carries a holomorphic geometric structures, generalizing the case of Hopf surfaces achieved in \cite{MP2010}. 

We also mention manifolds obtained as quotients of $\mathrm{SL}(2,\C)$ by cocompact lattices, which are examples of compact complex non-Kählerian and parallelizable manifolds, meaning that their holomorphic tangent bundle is trivial. Their deformation spaces were studied by Ghys in \cite{Ghys95}, and the techniques involved in this paper are a source of inspiration for our work, particularly in Section \ref{Subsection Kuranishi family of a resonant LVM manifold}. To study the deformations of compact quotients of $\mathrm{SL}(2,\C)$, Ghys first describes the deformation space of some locally homogeneous geometric structures, $(G,X)$-structures (see Section \ref{Subsection G,X structures}), that actually induce projective structures. He shows that the space of deformations of the complex structure of a compact quotient of $\mathrm{SL}(2,\C)$ is induced by the deformation space of the geometric structure, and shows an analogue of the Uniformization Theorem in this context. Although there exists no generalization of the Uniformization Theorem in dimension greater than one, compact quotients of $\mathrm{SL}(2,\C)$ form an example of family of compact complex manifolds that is uniformizable, in the sense of the Uniformization Theorem.

\subsection{LVM manifolds}

During his thesis \cite{MeThese} (see also \cite{Mee2000}), Meersseman constructed a class of compact complex manifolds of any dimension, generalizing the construction of L\'opez de Medrano and Verjovsky presented in \cite{LdmV97}. These manifolds are called \textbf{LVM manifolds} and this paper is devoted to the study of the complex geometry of some LVM threefolds.

We mention the work of Bosio, in \cite{Bos2001}, who constructed new compact complex manifolds, generalizing the construction of LVM manifolds and extending this family to the family of LVMB manifolds. However, in this work, we will only deal with LVM manifolds, and therefore, we will not give details on the construction of Bosio.

For most of them, LVM manifolds do not have the topology of a symplectic manifold and, hence, do not carry any Kähler metric, implying that they are not projective. It is worth noticing that it is impossible, in the holomorphic context, to construct a relevant compact submanifold of $\C^n$, and that the closed holomorphic submanifolds of the projective space $\PP^n(\C)$ are, by Chow's theorem (\cite{Chow49}, Theorem V p.910), projective. It is therefore enriching to have a construction of a family of non-projective manifolds in any dimension. Also, the world of non-Kählerian manifolds is both larger and more mysterious than that of Kählerian manifolds. For example, while the fundamental group of a Kählerian manifold has strong restrictions (see e.g. \cite{Amo96}), a corollary of a theorem by Taubes in \cite{Tau92} (p.165) guarantees that any group with a finite presentation is the fundamental group of a compact complex manifold of dimension $3$. It is also proven in \cite{Mee2000} (Theorem 3 p.89) that, except from compact tori, the LVM manifolds are not in the class $\mathcal{C}$ of Fujiki, i.e. they are not bimeromorphic to a compact Kähler manifold.

Through the construction of LVM manifolds, one can recover already well-known compact complex manifolds that we mentioned above. First, Meersseman shows that all compact complex tori are part of the LVM class, and he specifies which cases of his construction yield compact tori. To illustrate our previous remark, the compact tori are the only LVM manifolds that are symplectic. Next, the LVM class also includes the diagonal Hopf manifolds, that are obtained as the quotient of $\W$ by the action generated by a complex linear diagonal contracting map.

In \cite{Mee2000}, Meersseman studies geometric properties of LVM manifolds. First, he studies meromorphic functions and shows that generically, an LVM manifold does not carry any non-constant meromorphic function. Then, he computes the dimension of the Dolbeault cohomology space of global holomorphic $1$-forms and the one of global holomorphic vector fields in the generic case. Still in the generic case, every submanifold of an LVM manifold is given by a sub-LVM configuration (see Definition \ref{Definition Siegel domain and weak hyperbolicity}). 

Another important geometric property of LVM manifolds is that they carry a natural regular holomorphic transversely Kählerian foliation. Under a rational condition (condition (K) in \cite{MeVe2004}), Meersseman and Verjovsky obtained in \cite{MeVe2004} that the leaves space of this foliation is a projective toric variety equipped with an equivariant orbifold structure turning the projection from the LVM manifold to the leaves space into a principal Seifert bundle with typical fiber a compact complex torus. In the irrational case, the leaf space is equipped with a richer structure of Quantum toric stack as developed in \cite{KLMV} and \cite{KLMV2}.

For more bibliography that involve LVM manifolds, the reader can consult \cite{CFZ2007}, \cite{Tam2012}, \cite{PU2012}, \cite{PU2012}, \cite{BZ2015}, \cite{PUV2016}, \cite{BZ2017}, \cite{BPZ2019}, \cite{Ish2019}, \cite{KP2021} and \cite{IKP2022}

\subsection{Deformations of the complex structure of an LVM manifold}

In \cite{Mee2000}, Meersseman specifically studies the deformations of the complex structure of an LVM manifold. The construction of an LVM manifold requires a choice of complex parameters and thus naturally provides a family of deformations in the sense of the theory of Kodaira, Spencer, and Kuranishi (see \cite{KS58} \cite{Kur62}).

An LVM manifold is constructed starting from the data of a family of $m$ holomorphic linear diagonal vector fields of $\C^n$ (see Section \ref{Subsection Definition of LVM manifolds}). Such a family of vector fields is given by $n\times m$ coefficients, ordered in $m$ vectors of $\C^n$ (a configuration, see Definition \ref{Definition Siegel domain and weak hyperbolicity}), and allowing these coefficients to move gives the base space of a family of LVM manifolds. Two configurations that differ by an affine transformation of $\C^n$ give rise to LVM manifolds that are isomorphic, therefore, the base space of what we will call the LVM family is obtained by taking the quotient by the action of the affine group. Meersseman shows in \cite{Mee2000} (Lemma VI.1 p.102), that this space is naturally homeomorphic to an open set of $(\C^{n-m-1})^m$. He also shows (Theorem 11 p.103) that, under certain strong assumptions on the topology of the LVM manifold, his family of deformations is, at every point of its base, the Kuranishi family of the corresponding LVM manifold (see Theorem \ref{Theorem Deformation de Meersseman}). It is important to note that this base space, and thus this Kuranishi space, is smooth. However this leaves wide open the description of deformations for almost all topological types of LVM manifolds. One of the interests of this paper is to deal with deformations of some LVM threefolds that do not satisfy the hypotheses of Theorem \ref{Theorem Deformation de Meersseman}.

\subsection{LVM manifolds with holomorphic geometric structures}

Compact complex tori and diagonal Hopf manifolds are part of a larger subclass of LVM manifolds that are of interest in the context of geometric structures. As we present in Section \ref{Subsection Definition of LVM manifolds}, the construction of an LVM manifold depends on the choice of some complex parameters. Associated with this family of parameters, there is the concept of indispensable points, which we will not detail here (see Definition \ref{Definition Indispensable points}). The purpose that we want to emphasize here is that under the assumption that there are enough indispensable points among the parameters, the associated LVM manifold can be described as the quotient of an open set of $\C^k$, for some $k$, by a fixed-point free and proper action of $\Z^m$ generated by linear diagonal maps (in the standard coordinates). It is important to note the similarity with the definition of Hopf manifolds. This description shows that the LVM manifold carries a flat and torsion-free affine connection, i.e. a $(\mathrm{GA}(\C^k),\C^k)$-structure (see Section \ref{Subsection G,X structures}) or an affine structure (see \cite{MeeBos2006} Lemma 12.1 p.119). We will call this structure the canonical structure. We will see in Section \ref{Section Resonant structures} that this canonical structure can be seen differently, here as another type of holomorphic $(G,X)$-structure. The latter will be called the canonical resonant structure of the LVM manifold. 

The theory of $(G,X)$-structure, which we will partially recall in Section \ref{Subsection G,X structures} comes with a local deformation theory for compact smooth $(G,X)$-manifolds, as well as the theory of complex structures. The main result of this theory is the so-called Ehresmann-Thurston principle (see Theorem \ref{Theorem Ehresmann Thurston principle}). The fundamental difference with the theory of local deformations of complex structures, is that, locally, deforming a $(G,X)$-structure is the same thing as deforming a group homomorphism. Hence, we deal with global objects, and not an infinitesimal theory. 

The space of local smooth deformations of a holomorphic $(G,X)$-structure and the space of local deformations of a complex structure are related in the following sense. A smooth manifold equipped with a $(G,X)$-structure, where $G$ is a complex Lie group acting holomorphically on a complex manifold $X$, inherits naturally a complex structure. In addition to being interesting in itself, the deformation space (in the smooth category) of a holomorphic $(G,X)$-structure therefore provides insights into the Kuranishi space of the underlying compact complex manifold.

\subsection{LVM manifolds of type (2,6,4)}

The classification of LVM manifolds of complex dimension $1$ is rather easy, since they are all compact tori of dimension $1$. For them, the geometry and the deformation theory is well known via the Uniformization Theorem, even if it represents a very special case -- real surfaces of genus $1$ -- for which the central tool is the theory of elliptic functions. Elliptic curves are the only compact Riemann surfaces with an affine structure.

In complex dimension $2$, an LVM manifold is a compact torus of dimension $2$, or a Hopf surface. Once again, their small deformations are well known (see \cite{KS58}, Theorem 14.3 p.410 for tori and \cite{Ha85}, Theorem A.3 p.249 for Hopf manifolds), and all these complex surfaces are geometric, in the sense that their complex structure are induced by a $(G,X)$-structure. Compact tori are affine and Hopf surfaces carry the $\mathcal{O}(n)$-structures constructed in \cite{MP2010}.

The world of LVM manifolds of dimension $3$ is more complicated. An LVM manifold of dimension $3$ can be a compact torus of dimension $3$, an LVM manifold that we shall call \textbf{of type $(2,6,4)$} (see Section \ref{Subsection Type (2,6,4)}), a Hopf threefold, some deformation of Calabi-Eckmann threefold (with underlying topology $\mathbb S^3\times \mathbb S^3$, see \cite{CE53}), or an LVM manifold, say of type $(1,5,0)$, with complicated topology (the quotient of a connected sum of five copies of $\mathbb S^3\times \mathbb S^4$ by a nontrivial action of $\mathbb S^1$, see \cite{MeeBos2006}, Example 5.1. p.86). The Kuranishi spaces of tori, Hopf manifolds and Calabi-Eckmann manifolds are well known (see the previous references for tori and Hopf threefolds, and \cite{Ak75} for Calabi-Eckmann manifolds). Calabi-Eckmann manifolds and $(1,5,0)$ LVM manifolds do not fall in the case where the LVM manifold has a canonical geometric structure mentioned above, while tori and Hopf threefolds are geometric (\cite{Bour2025} and \cite{Mad25}). The Kuranishi space of a general LVM-deformation of a Calabi-Eckmann manifold remain unknown (see e.g. \cite{LN96} Remark 6 p.800).

The first example of LVM manifold that falls into our context of interest is therefore the type $(2,6,4)$ (see Section \ref{Subsection Type (2,6,4)} for the required details on these manifolds). Namely, they have the property of carrying a canonical resonant structure (see above) and their deformations as compact complex manifolds are not known. LVM manifolds of type $(2,6,4)$, and the study of the deformations of their canonical geometric structures and their interplay (resonant and complex structures) are the objects of this paper. An LVM manifold of type $(2,6,4)$ is obtained as a quotient of $V:=\C^*\times \C^2\setminus\{0\}$ by an action of the group $\Z^2$ generated by two linear diagonal mappings of $\C^3$. The underlying smooth manifold is isomorphic to the product of spheres $\mathbb S^1\times\mathbb S^1\times\mathbb S^1\times \mathbb S^3$. 

\subsection{Description of the results}

We would like to describe and compare different deformation spaces. As mentioned above, the natural LVM family furnishes a candidate to describe the Kuranishi family of every LVM manifold, and we will prove that the LVM family is the Kuranishi family of a \textbf{generic} LVM manifold of type $(2,6,4)$. However, some LVM manifolds carry resonances (see Definition \ref{Definition Resonances for an LVM manifold}), whose definition is similar to the one for Hopf manifolds (see \cite{Bo81} p.1). For the resonant LVM manifolds, the LVM family lacks complex structures. We describe the deformation space of the canonical resonant structure -- a $(G_N,V)$-structure with $G_N$ the complex Lie group of resonant transformations of $V$ (see Section \ref{Subsection Group of resonant transformations}) -- and we show that for every LVM manifold of type $(2,6,4)$, its Kuranishi family is induced by the deformation space of resonant structures. 

\subsubsection*{Deformation of the canonical resonant structure}

Our principal result concerning the deformations of the canonical resonant structure is the first part of the above mentioned Theorem \ref{TheoremIntro Action and Kuranishi}. One interpretation of this result is the following. A holomorphic, free, proper, and cocompact action endows the compact quotient it defines with a complex structure. Homomorphisms such that their image is constituted of maps that are diagonalizable in a change of coordinates of $V$ give rise to LVM manifolds, so the result is well known in that case. Other homomorphisms, that appear in the case where $N$ is resonant in the sense of Definition \ref{Definition Resonances for an LVM manifold}, provide new complex structures on the product of spheres $\mathbb S^1\times \mathbb S^1\times \mathbb S^1\times \mathbb S^3$. 

On the other hand, this result allows for the description of $(G_N,V)$-structures that are close to the canonical structure associated with an LVM manifold and for which the covering space $V=\C^*\times \C^2\setminus\{0\}$ has trivial holonomy. Indeed, the developing map of such a structure is, according to this theorem, the universal covering map of $V$, and the structure is complete (see Definition \ref{Definition completeness and Kleinian for GX structures}). 

We also describe all the other nearby $(G_N,V)$-structures such that $V$ does not have a trivial holonomy. Their developing map is computed in the proof of Lemma \ref{Lemma map between the representations varieties for Gpq}. None of these structure are uniformizable (see Definition \ref{Definition completeness and Kleinian for GX structures}).

\subsubsection*{Kuranishi family of a resonant LVM manifold}

Let $N$ be an LVM manifold of type $(2,6,4)$ and let $G_N$ be the group of resonant transformations of $V$ associated with the resonances of $N$. According to Theorem \ref{TheoremIntro Action and Kuranishi}, there is an open neighborhood of the canonical homomorphism of $N$ in $\Hom(\Z^2,G_N)$ such that the action of $\Z^2$ on $V$ generated by a homomorphism in this neighborhood is fixed point free, proper and cocompact (if $N$ is non-resonant, such an action is an LVM action). Therefore, there is a germ of family of deformations with base space $\Hom(\Z^2,G_N)$. Notice that this germ of space is smooth if and only if $N$ does not carry any non-trivial resonance. The second part of Theorem \ref{TheoremIntro Action and Kuranishi} states that this family of deformations of $N$ is its Kuranishi family.

In non-resonant case, this result can be reformulated as follows. The germ of the representation variety $\Hom(\Z^2,G_N)$ at the canonical homomorphism is isomorphic to the germ of base space of the LVM family at $N$, and the LVM family is the Kuranishi family of $N$. As a consequence of the cohomology computations in Proposition \ref{Proposition Cohomology in Theta for a LVM manifold of type 264}, this is false for a resonant LVM manifold.

The proof of this theorem is more difficult in the resonant case than in the non resonant case. In the latter, the LVM family has a smooth base space, and therefore, showing that the Kodaira-Spencer map of the family is an isomorphism is sufficient to show that the family is the Kuranishi family of $N$. In the resonant case, the Kuranishi family has a singular base space. This fact is already implied by the computation of the first obstruction to deformations achieved in Proposition \ref{Proposition First obstruction}. By the way, the simple existence of the family with base $\Hom(\Z^2,G_N)$ shows that all the other obstructions vanish (see Corollary \ref{Corollary The higher obstructions vanish}). To prove that this family is the Kuranishi family of $N$, we use the point of view of obstructions proposed by Douady in \cite{Dou60Obs} and the argument of Ghys developed in \cite{Ghys95}, who studied the deformations of compact quotients of $\mathrm{SL}(2,\C)$ by discrete subgroups. As in Ghys' paper, we start by using the Ehresmann-Thurston principle on some $(G,X)$-structures, then we show a sort of uniformization theorem (see \cite{Ghys95}, Lemme 2.1 p.115) to describe these structures and finish by showing that the natural family associated with the representation variety is the Kuranishi family of the associated compact manifold. 

This result also contains the description of the Kuranishi family of products of some complex manifolds. Indeed, if we consider a diagonal Hopf surface and an elliptic curve, their product is naturally seen as an LVM manifold of type $(2,6,4)$ and Theorem \ref{TheoremIntro Action and Kuranishi} describes its Kuranishi family. One important remark is that we recover among this family all the products formed by an Hopf surface (non-necessarily diagonal) and an elliptic curve. On the other hand there exist arbitrary small deformations of such a product that do not have the complex structure of a product.

\subsubsection*{Geometrization of LVM manifolds}

Once we have described the local deformation space of the canonical resonant structure and the Kuranishi family of an LVM manifold $N$ of type $(2,6,4)$, we can describe the germ of map 
\small
$$
\{N\text{-resonant structure on } \mathbb S^1\times\mathbb S^1\times\mathbb S^1\times\mathbb S^3\}\longrightarrow \{\text{complex structure on } \mathbb S^1\times\mathbb S^1\times\mathbb S^1\times\mathbb S^3\},
$$
\normalsize
at the canonical structure. In this context, the best deformation space that we can use is the representation variety $\Hom(\Z^3,G_N)$, where $G_N$ is the group of $N$-resonant transformations of $V=\C^*\times \C^2\setminus\{0\}$. Of course, the orbits of the action of $G_N$ acting by conjugation on $\Hom(\Z^3,G_N)$ contain $(G_N,V)$-structures which are isomorphic, nevertheless the character variety is not Hausdorff, except in the non-resonant case, where the group $G_N$ is abelian, so the character variety is $G_N^3$, which is smooth. Our main result in this context is Theorem \ref{TheoremIntro Geometrization}.

\subsubsection*{Complete family of LVM manifolds}

In \cite{Dab82}, Dabrowski suggested a method for gluing families of Hopf surfaces to construct a deformation family that encompasses all isomorphism classes of Hopf surfaces. One important point to note is that, among Hopf manifolds, this approach is only feasible for surfaces. In fact, the class of Hopf manifolds with dimension $n \geq 3$ has the characteristic that the dimension of the Kuranishi space is not bounded. This is a consequence of the fact that, for Hopf manifolds of dimension $n \geq 3$, the number of resonances can become arbitrarily large. Although Dabrowski's family includes all Hopf surfaces, it does not fulfill the condition of being complete (as explained in Theorem \ref{Theorem Kuranihi}).

In his thesis \cite{Fro2017}, Fromenteau revisits Dabrowski's method and succeeds in constructing a family of Hopf surfaces that contains all of them and is complete at every point.

LVM manifolds of type $(2,6,4)$ share similar properties with Hopf surfaces, like having a bounded number of resonances (see Definition \ref{Definition Resonances for an LVM manifold}). Consequently, it seems reasonable to expect that we can glue the Kuranishi families of LVM manifolds to form a family that includes all LVM manifolds of type $(2,6,4)$ and is complete at each point. We draw upon Fromenteau’s construction to develop such a family.

\begin{TheoremIntro}\label{TheoremIntro Big family of LVM}
    (Theorem \ref{Theorem big family of LVM manifolds}) There exists a family of deformations that contains all LVM manifolds of type $(2,6,4)$ and that is complete at every such a manifold.
\end{TheoremIntro}

\section{LVM manifolds}\label{Section LVM manifolds}

\subsection{Definition}\label{Subsection Definition of LVM manifolds}

We recall here the construction of LVM manifolds. The main reference is \cite{Mee2000}. Let $m$ and $n$ be two integers that satisfy $n\geq 2m+1$ and $m\geq 1$. We consider $(\Lambda_1,...,\Lambda_n)$ a $n$-tuple of vectors in $\C^m$, that we will call a \textbf{configuration} and we generate an action of $\C^m$ on $\C^n$ defined by 
$$
\begin{matrix}
    \C^m\times \C^n & \longrightarrow & \C^n \\
    (T,(z_1,...,z_n)) & \longmapsto & (\exp(<\Lambda_1,T>z_1),...,\exp(<\Lambda_n,T>z_n)),
\end{matrix}
$$
where $<.,.>$ is the standard complex bilinear form (not hermitian) of $\C^n$.

\begin{Definition}\label{Definition Siegel domain and weak hyperbolicity}
    (\cite{Mee2000}, p.82) The configuration $(\Lambda_1,...,\Lambda_n)$ of vectors in $\C^m$ fulfills
    \begin{itemize}
        \item the \textbf{Siegel condition} if the origin $0$ in $\C^m$ belongs to the convex hull of the vectors $\Lambda_1,...,\Lambda_n$,
        \item the \textbf{weak hyperbolicity condition} if for every sub-family of $2m$ vectors 
        $$
        \Lambda_{j_1}, ..., \Lambda_{j_{2m}},
        $$ 
        the origin does not belong to the convex hull generated by those $2m$ vectors.
    \end{itemize}
    
    A configuration satisfying both of those properties is said to be an \textbf{LVM configuration}.
\end{Definition}

\begin{Definition}\label{Definition Indispensable points}
    (\cite{Mee2000}, p.86) Let $\Lambda=(\Lambda_1,...,\Lambda_n)$ be an LVM configuration of vectors in $\C^m$ and let $j$ be an element of $\{1,...,n\}$. We say that $\Lambda_j$ is an \textbf{indispensable point} if $0$ does not belong to the convex hull of $\{\Lambda_l,~l\neq j\}$. 

    Such an LVM configuration with $k$ indispensable points is called an LVM configuration of type $(m,n,k)$.
\end{Definition}

From now on, we assume that the configuration $\Lambda:=(\Lambda_1,...,\Lambda_n)$ is LVM of type $(m,n,k)$. For $z=(z_1,...,z_n)$ a point in $\C^n$, we denote by $I_z$ the set 
$$
I_z:=\{j\in \{1,...,n\},~z_j\neq0\}.
$$
We consider the open subset $S$ of $\C^n$ defined as follows:
\small
$$
S:=\{z=(z_1,...,z_n)\in\C^n\setminus\{0\},~0 \text{ belongs to the convex hull of }\{\Lambda_j,~j\in I_z\}\}.
$$
\normalsize

It is proved in \cite{Mee2000} that $S$ contains $(\C^*)^n$ and is therefore dense in $\C^n$, that $S$ is the union of the closed orbits of the previous action of $\C^m$ on $\C^n$, and that the set of closed orbits, i.e. $S/\C^m$, is homeomorphic to the differential submanifold of $S$ defined by 
$$
\mathcal T:=\left\{z=(z_1,...,z_n)\in S,~\sum_{j=1}^n\vert z_j\vert^2\Lambda_j=0\right\}.
$$

We denote by $\pi$ the canonical projection $\C^n\setminus\{0\}\to \PP^{n-1}(\C)$. This is a holomorphic $\C^*$-principal bundle and the dilation action of $\C^*$ on $\C^n\setminus\{0\}$ commutes with the action of $\C^m$ given by the configuration $\Lambda$. Therefore, the action of $\C^m$ descends to an action on $\PP^{n-1}(\C)$. It is proved in \cite{Mee2000} that $\PP (S)=\pi(S)$, is stable under the action of $\C^m$ and that the quotient space $\PP (S)/\C^m$ is homeomorphic to the differential submanifold of $\PP(S)$ defined by 
$$
\mathcal N:=\left\{z=[z_1,...,z_n]\in \PP(S),~\sum_{j=1}^n\vert z_j\vert^2\Lambda_j=0\right\}.
$$

This submanifold of $\PP(S)$ is transverse to the foliation given by the orbits of the action, so $\mathcal N$ and, therefore, the quotient space $\PP(S)/\C^m$ inherit a complex structure given by the transverse part of a foliated atlas. 

The compact complex manifold thus created is called the \textbf{LVM manifold} associated with the configuration $\Lambda$ and will be denoted by $N_\Lambda$ or $N$ if there is no ambiguity about the configuration from which it comes.

The following is proved in \cite{Mee2000}.

\begin{Proposition}\label{Proposition les tores compacts sont LVM}
    (\cite{Mee2000}, p.86) Let $\Lambda$ be an LVM configuration of type $(m,n,k)$. Then, $k$ is bounded above by $2m+1$ and $k=2m+1$ if and only if $n=2m+1$.
    
    Let $\Lambda$ be an LVM configuration of type $(m,2m+1,2m+1)$. Then, $N_\Lambda$ is biholomorphic to a compact complex torus of dimension $m$.

    Moreover, for every given compact complex torus, there exists an LVM configuration such that the associated LVM manifold is biholomorphic to the given torus.
\end{Proposition}

As a consequence of Carathéodory's Theorem and Proposition \ref{Proposition les tores compacts sont LVM}, every LVM manifold of type $(m,n,k)$ has a submanifold biholomorphic to a compact complex torus of complex dimension $m$. 

\subsection{A deformation space}\label{Subsection A deformation space}

When studying the construction of LVM manifolds, it is natural to wish to construct a deformation space by allowing the parameters of the construction, i.e. the LVM configuration $\Lambda$, to vary. This space will be enhanced by taking into account an equivalence relation that describes simple biholomorphisms between the different manifolds.

\begin{Definition}\label{Definition Equivalent configurations LVM}
    (\cite{Mee2000}, p.102) Let $\Lambda$ and $\Lambda'$ be LVM configurations of type $(m,n,k)$ and $(m,n,k')$. We say that $\Lambda$ and $\Lambda'$ are \textbf{equivalent} if there is a differentiable map $H:[0,1]\to(\C^m)^n$ such that 
    \begin{itemize}
        \item $H(0)=\Lambda$,
        \item $H(1)=\Lambda'$,
        \item for every $t$ in $[0,1]$, $H(t)$ is a LVM configuration.
    \end{itemize}
\end{Definition}

\begin{Proposition}\label{Proposition Two equivalent configurations have the same type}
    (\cite{Mee2000}, p.102) Let $\Lambda$ and $\Lambda'$ be two equivalent LVM configurations of type $(m,n,k)$ and $(m,n,k')$. Then, they have the same type, i.e. $k=k'$.
\end{Proposition}


\begin{Definition}\label{Definition Affine Equivalent configurations LVM}
    (\cite{Mee2000}, p.102) Let $\Lambda$ and $\Lambda'$ be equivalent LVM configurations of type $(m,n,k)$. We say that they are \textbf{affine equivalent} if there exists a complex affine transformation of $\C^m$ sending $\Lambda$ to $\Lambda'$. 
\end{Definition}

Let $\Lambda$ be an LVM configuration of type $(m,n,k)$ and let $D$ be the space of LVM configurations that are equivalent with $\Lambda$, where we identify two of them if they are affine equivalent. Then the following is proved in \cite{Mee2000}.

\begin{Lemma}\label{Lemma The space of defo of Laurent is an open set}
    (\cite{Mee2000}, Lemma VI.1. p.102) The space $D$ is naturally homeomorphic with an open subset of $(\C^{n-m-1})^m$ and, therefore, is equipped with a natural complex structure of dimension $m(n-m-1)$.
\end{Lemma}

We recall that, for an LVM configuration $\Lambda$, the open subset $S$ of $\C^n$ is defined by
\small
$$
S:=\{z=(z_1,...,z_n)\in\C^n\setminus\{0\},~0 \text{ belongs to the convex hull of }\{\Lambda_j,~j\in I_z\}\}.
$$
\normalsize
Let $d$ be the minimal complex codimension of $E$ where $S=\C^n\setminus E$. 

\begin{TheoremLetter}\label{Theorem Deformation de Meersseman}
    (\cite{Mee2000} Theorem 11 p.103) Let $\Lambda$ be an LVM configuration, let $N$ be the associated LVM manifold and let $D$ be the associated deformation space. Assume that $d>3$ and that all the $\Lambda_i$ are distinct, then $D$ is a universal deformation space of $N$.
\end{TheoremLetter}

This result is essentially the only deformation result on LVM manifolds. The assumptions of this result imply in particular that there are no indispensable points. This framework is the opposite of the one that is treated in this article. As we will see in the next section, we will focus on LVM manifolds having a high number of indispensable points.

\subsection{High number of indispensable points}\label{Subsection High number of indispensable points}

Let $\Lambda=(\Lambda_1,...,\Lambda_n)$ be an LVM configuration of type $(m,n,k)$ with $k\geq m+1$ indispensable points, that we can assume to be $\Lambda_1,...,\Lambda_k$. Then, the open subset $S$ of $\C^n$ that appears in the construction of the LVM manifold associated to $\Lambda$ is of the form $(\C^*)^{m+1}\times V$, where $V$ is an open subset of $\C^{n-m-1}$. 

It is proved in \cite{MeeBos2006} (Lemma 12.1) that the LVM manifold associated with the configuration $(\Lambda_1,...,\Lambda_n)$ is isomorphic to a quotient of the open set $V$ by a linear diagonal action of $\Z^m$ as follows. We consider $\Omega$ the $m\times m$ invertible matrix where the rows are given by the vectors $\Lambda_{j+1}-\Lambda_1$ for $j\in\{1,...,m\}$. Then, the action of $\Z^m$ is given by the group representation

\begin{center}
    $
    \begin{matrix}
        \mathbb{Z}^m & \longrightarrow & \mathrm{Diag}(\mathbb{C}^{n-m-1})^{\times} \\
        K & \longmapsto & A_K := \mathrm{Diag} (\alpha_{1,K},...,\alpha_{n-m-1,K}) ,
    \end{matrix}
    $
\end{center}
where 
$$
\alpha_{j,K}=\exp(2i\pi<\Lambda_{m+1+j}-\Lambda_1,\Omega^{-1}K>).
$$

The action is fixed point free and proper, so the quotient map is a covering map (the universal covering map if and only if $k=m+1$, i.e. $V$ is simply connected) and the quotient inherits a canonical complex structure from $V$. It should be noted that the complex manifold obtained this way is the LVM manifold associated with $\Lambda$.

With this description, an LVM manifold with at least $m+1$ indispensable points is naturally endowed with a flat torsion-free holomorphic affine connection. In Section \ref{Subsection G,X structures}, we will recall the concept of $(G,X)$-structure. In this language, an LVM manifold with at least $m+1$ indispensable points is naturally endowed with a holomorphic $(\mathrm{GA}(\C^n),\C^n)$-structure (see Example \ref{Example of G,X struct}).

We will now give a lemma that will highlight some properties of the deformation space $D$ introduced in Section \ref{Subsection A deformation space} in this context of a high number of indispensable points.

\begin{Lemma}\label{Lemma Deformation space in the case of k plus grand que m+1}
    Let $\Lambda=(\Lambda_1,...,\Lambda_n)$ be an LVM configuration of type $(m,n,k)$ with $k\geq m+1$. We assume, without loss of generality, that $\Lambda_1,...,\Lambda_{m+1}$ are indispensable points. Then, the deformation space $D$ is biholomorphic to 
    \footnotesize
    $$
    \widetilde{D}:=\{(\Lambda_{m+2}',...,\Lambda_n')\in(\C^m)^{n-m-1},~(\Lambda_1,...,\Lambda_{m+1},\Lambda_{m+2}',...,\Lambda_n') \text{ is equivalent to }\Lambda\}.
    $$
    \normalsize
\end{Lemma}

\begin{proof}\label{Preuve Lemma Deformation space in the case of k plus grand que m+1}
    According to Carathéodory Theorem and the fact that $\Lambda_1,...,\Lambda_{m+1}$ are indispensable, we can assume that $(\Lambda_1,...,\Lambda_{2m+1})$ is an LVM configuration. By \cite{MeThese} Lemme 2 p.15, the family of vectors $\{\Lambda_1,...,\Lambda_{m+1}\}$ is complex affinely independent. Therefore, there exists a unique complex affine transformation $f$ of $\C^m$ sending $\Lambda_1,...,\Lambda_{m+1}$ on $e_1,...,e_m,0$, where $(e_1,...,e_m)$ is the standard basis of $\C^m$. 
    
    We define the holomorphic map 
    $$
    \begin{matrix}
        \phi&:&\widetilde{D} & \longrightarrow & D \\
        &&(\Lambda_{m+2}',...,\Lambda_n') & \longmapsto & [\Lambda_1,...,\Lambda_{m+1},\Lambda_{m+2}',...,\Lambda_n'].
    \end{matrix}
    $$
    Let $\Lambda''$ be an LVM configuration equivalent to $\Lambda$. According to Proposition \ref{Proposition Two equivalent configurations have the same type}, $\Lambda_1'',...,\Lambda_{m+1}''$ are indispensable points of the configuration $\Lambda''$. Therefore, as before, there exists an unique complex affine transformation $g$ of $\C^m$ sending $\Lambda_1'',...,\Lambda_{m+1}''$ on $e_1,...,e_m,0$. Defining, for $j\in\{m+2,...,n\}$, 
    $$
    \Lambda_j':=f^{-1}(g(\Lambda_j)),
    $$
    we have, in $D$,
    $$
    [\Lambda'']=[\Lambda_1,...,\Lambda_{m+1},\Lambda_{m+2}',...,\Lambda_n']=\phi(\Lambda_{m+2}',...,\Lambda_n').
    $$
    We infer that $\phi$ is onto. Since the family $\{\Lambda_1,...,\Lambda_{m+1}\}$ is complex affinely independent, $\phi$ is also one-to-one. Finally, we get that $\phi$ is a biholomorphism.
\end{proof}

In the case of an LVM configuration $\Lambda$ of type $(m,n,k)$ with $k\geq m+1$, we will always assume that $\Lambda_1,...,\Lambda_k$ are indispensable points and we will identify $D$ with the $\widetilde{D}$ of Proposition \ref{Lemma Deformation space in the case of k plus grand que m+1}.

\subsection{Type (2,6,4)}\label{Subsection Type (2,6,4)}

The cases where the parameter $m$ is equal to $1$ and the number of indispensable point $k$ satisfies $k\geq 2$ leads to elliptic curves if $n=3$ and Hopf manifolds if $n\geq4$. The Kuranishi family of such a complex manifold is already well studied. Moving to the case $m=2$ leads first to the case of compact tori of dimension $2$, and then to the case $n=6$. Therefore, we focus on this case, i.e. the case $m=2$ and $n=6$. It is important to note that, for combinatorial reasons revealed in \cite{Mee2000}, p.86, this case forces the parameter $k$ to take the value $4$. Following the previous section \ref{Subsection High number of indispensable points}, LVM manifolds of type $(2,6,4)$ are obtained by quotients of $V=\C^*\times \C^2\setminus\{0\}$ under some $\Z^2$ action by complex linear diagonal maps. The differentiable structure of an LVM manifold of type $(2,6,4)$ is isomorphic to the product of spheres $\mathbb S^1\times\mathbb S^1\times\mathbb S^1\times\mathbb S^3$.

We give here some lemmas that will be useful in the study of these manifolds. The first one is a straightforward reformulation of \cite{MeThese} Lemme 1 p.36, and the second and third one are consequences of the first.

\begin{Lemma}\label{Lemma Adaptation des Lambdas au cas 2 6 4}
    Let $(\Lambda_1,...,\Lambda_{5})$ be an LVM configuration of type $(2,5,5)$. Then for every $j_1<j_2\in \{2,...,5\}$, we have 
    $$
    \rk_\mathbb C\{\Lambda_{j_k}-\Lambda_1,~k=1:2\}=2.
    $$
\end{Lemma}

\begin{Lemma}\label{Lemma Useful lemma 2 6 4}
    Let $\alpha_1,\alpha_2,\alpha_3$ and $\beta_1,\beta_2,\beta_3$ be as above. Then
    \begin{enumerate}[label=(\roman*)]
        \item For $j\in \{1,2,3\}$, we can't have $\vert\alpha_j\vert=\vert\beta_j\vert=1$,
        \item for $j\in \{2,3\}$, we can't have $\vert\alpha_1\vert=\vert\alpha_j\vert=1$ or $\vert\beta_1\vert=\vert\beta_j\vert=1$.
    \end{enumerate}
\end{Lemma}

\begin{Lemma}\label{Lemma Impossible affine resonances in the case LVM 264}
    Let $\alpha_1,\alpha_2,\alpha_3$ and $\beta_1,\beta_2,\beta_3$ be as above. Then for $j\in \{2,3\}$, we can't have $\alpha_1=\alpha_j$ and $\beta_1=\beta_j$.
\end{Lemma}

The goal of this paper is to study deformation spaces of geometric structures (complex structure and complex $(G,X)$-structures, see next section \ref{Section Deformations of geometric structures}) on LVM manifolds of type $(2,6,4)$.

\section{Deformations of geometric structures}\label{Section Deformations of geometric structures}

\subsection{Complex structures}\label{Subsection Deformations of the complex structure}

The aim of this section is to recall the Kuranishi theorem, which is the central result in the theory of deformations of complex manifolds. We refer to \cite{KodBook}, \cite{MoKo2006} and \cite{Kur62}, or the lecture notes \cite{Cat83}, if the reader is interested in the Kodaira-Spencer-Kuranishi theory of deformations. 

\begin{TheoremLetter}\label{Theorem Kuranihi}
    (\cite{Kur62} Theorem 2. p.571) Let $M$ be a compact complex manifold. There exist an analytic space $K_M$, called the \textbf{Kuranishi space} of $M$, a point $o\in K_M$ and a family of compact complex manifolds $\varpi_M:\mathcal M\to K_M$, called the \textbf{Kuranishi family} of $M$, parametrized by $K_M$ that is versal at $M=\varpi_M^{-1}(o)$, i.e. such that every family of deformation of $M$ is obtained by a pullback of the Kuranishi family, and the differential of the map giving the pullback is unique (and is the Kodaira-Spencer map of the family).

    Moreover, the Zariski tangent space of $K_M$ at the point that parametrizes $M$ is isomorphic to $H^1(M,TM)$ and there exists an open subset of $K_M$ where the family is complete at every point $o'\in K_M$, i.e. every family of deformations of $M'=\varpi_M^{-1}(o')$ is obtained by a pullback of $\varpi_M$.
\end{TheoremLetter}

The following result is an immediate corollary of the Kuranishi theorem \ref{Theorem Kuranihi}.

\begin{Corollary}\label{Corollary Smooth family with invertible KS map}
    Let $M$ be a compact complex manifold and let $\varpi:\mathcal M\to B$ be a family of deformations of $M=\varpi^{-1}(t)$, with $B$ a smooth complex manifold. Assume that the Kodaira-Spencer map of $\varpi$ at $t$ is invertible. Then, $\varpi$ is the Kuranishi family of $M$.
\end{Corollary}

\subsection{(G,X)-structures}\label{Subsection G,X structures}

In this section, we recall the definitions and properties of $(G,X)$-structures. We refer mainly to \cite{Ehr36}, \cite{Thu} and \cite{Gold88}.

Let $G$ be a complex Lie group, and let $X$ be a complex manifold endowed with a holomorphic action of $G$. Let $M$ be a differentiable manifold, let $p:\widetilde{M}\to M$ be the universal covering map of $M$ and let $\pi_1(M)$ be the corresponding fundamental group.

\begin{Definition}\label{Definition G,X structures}
    A \textbf{$(G,X)$-structure} on $M$ is a pair $(\mathrm{Dev},\rho)$, where $\rho:\pi_1(M)\to G$ is a group homomorphism and $\mathrm{Dev} :\widetilde{M}\to X$ is a $\rho$-equivariant local diffeomorphism. 
\end{Definition}

Let $M$ be a manifold equipped with a $(G,X)$-structure, then $M$ is locally modeled on $X$ and has special local automorphisms, called \textbf{symmetries}, which correspond to the action of $G$ on $X$.

\begin{Remark}\label{Remark On the fact that a G,X manifold has an induced complex structure}
    The complex structure on $X$ and the fact that the action of $G$ is holomorphic induce on $M$ a complex structure. Therefore, there is a map 
    $$
    \{(G,X)-\text{structures on }M\}\longrightarrow \{\text{complex structures on }M\}.
    $$
    In the context of Riemann surfaces, this map has been essential in the proof of the Uniformization theorem. Indeed, with $G=\mathrm{PGL}(2,\C)$ and $X=\PP^1(\C)$, for every smooth oriented surface $S$, this map is onto and is a holomorphic submersion (\cite{HPSG2010}).
\end{Remark}

\begin{Example}\label{Example of G,X struct}
    Assume that the model of the geometry is given as follows. We take $G=\mathrm{GA}(\C^n)$ and $X=\C^n$. Then a $(G,X)$-structure on a smooth manifold $M$ is equivalent with the data of a complex structure on $M$ and a flat torsion-free holomorphic affine connection on $M$.
\end{Example}

\begin{Definition}\label{Definition completeness and Kleinian for GX structures}
    (\cite{Kas2018}, p.2) Suppose that $M$ is a manifold endowed with a $(G,X)$-structure. The $(G,X)$-structure is said to be \textbf{uniformizable} or \textbf{kleinian} if $M$ is isomorphic as $(G,X)$-manifolds to a quotient of an open subset $U$ of $X$ by a discrete subgroup of $G$ acting on $U$ without fixed points and properly.
    
    The $(G,X)$-structure is said to be \textbf{complete} if $M$ is isomorphic as $(G,X)$ manifolds to a quotient of $X$ by a discrete subgroup of $G$ acting on $X$ without fixed points and properly, i.e. if the $(G,X)$-structure is uniformizable with $U=X$.
\end{Definition}

We are interested in the deformations of $(G,X)$-structures on a compact smooth manifold $M$. The main result is the Ehresmann-Thurston principle. It roughly speaking asserts that if $(\mathrm{Dev},\rho)$ is a $(G,X)$-structure on $M$, then for every homomorphism $\rho'$ close enough to $\rho$ in $\Hom(\pi_1(M),G)$ equipped with the compact open topology, there exists a $(G,X)$-structure on $M$ with holonomy $\rho'$. On the space of developing maps, we consider the $\mathcal C^\infty$ topology.

\begin{TheoremLetter}\label{Theorem Ehresmann Thurston principle}
    (\cite{Gold88}, Theorem 3.1 p.12) Let $M$ be a compact manifold and $(G,X)$ be a geometric model. The forgetful map 
    $$
    hol:(\mathrm{Dev},\rho) \longmapsto \rho
    $$
    that associates to a $(G,X)$-structure on $M$ its holonomy homomorphism is open. 
    
    Moreover, let $(\mathrm{Dev},\rho)$ be a $(G,X)$-structure on $M$. Then, there exists an open neighborhood $U$ of $(\mathrm{Dev},\rho)$ such that if $\rho'$ is in $hol(U)$ and if $(\mathrm{Dev}',\rho')$ and $(\mathrm{Dev}'',\rho')$ are elements of $U$, then there exists a diffeomorphism $\phi:\widetilde{M}\to \widetilde M$ and an element $g$ of the group $G$ such that $\mathrm{Dev}''=g\circ \mathrm{Dev}'\circ \phi$.
\end{TheoremLetter}

According to this principle, we will consider two deformation spaces of $(G,X)$-structures around a given one. The first one is the \textbf{representation variety} 
$$
\Hom(\pi_1(M),G).
$$ 
There is an action of the group $G$ on the space of couples $(\mathrm{Dev},\rho)$ constituted by a developing map $\mathrm{Dev}$ that is $\rho$-equivariant. For an element $g$ in $G$ and a pair $(\mathrm{Dev},\rho)$, the action is given by composing the developing map by the action of $g$ and by conjugating the homomorphism $\rho$ by $g$. It is straightforward that two $(G,X)$-structures in the same orbit are isomorphic. Therefore, it is interesting to consider the space of orbits under the action of $G$. However, as mentioned in \cite{Gold88}, it is difficult in general to control the orbits of the action of $G$ and the \textbf{character variety} 
$$
\Hom(\pi_1(M),G)/G
$$ has a complicated topology.

Depending on the context and the topology of the character variety around a base point, we will either consider the representation variety or the character variety as the deformation space.

\section{Infinitesimal deformations and LVM family}

This section is devoted to the study of infinitesimal deformations of the complex structure of an LVM manifold of type $(2,6,4)$. We will describe the cohomology in the holomorphic tangent sheaf in terms of the resonances relations (see \ref{Definition Resonances for an LVM manifold}). As a consequence of the computation of the cohomology, we will see that for a generic LVM manifold of type $(2,6,4)$, the deformation space $D$ considered at the end of Section \ref{Subsection High number of indispensable points} is the Kuranishi space of the LVM manifold.

\subsection{Cohomology in the tangent sheaf}\label{Subsection  Cohomology in the holomorphic tangent sheaf}

Diagonal Hopf manifolds are the first example (out of complex tori) of LVM manifolds with the condition $k\geq m+1$. 
Consider a Mall bundle (\cite{Ma96}) over a Hopf manifold $M=\C^n\setminus\{0\}/\Z$, i.e. a holomorphic vector bundle over $M$, with trivial pullback on $\C^n\setminus\{0\}$. The computation of the cohomology of the sheaf of sections can be achieved by using a long exact Douady sequence in cohomology given by the universal covering map. We can find this argument in \cite{Ma96}. 

We are interested in small deformations of the complex structure of an LVM manifold of type $(2,6,4)$. Therefore, we need to compute the crucial cohomology of the sheaf of holomorphic vector fields.

Let $(\Lambda_1,...,\Lambda_6)$ be an LVM configuration of type $(2,6,4)$. The LVM manifold $N$ associated with this configuration is the quotient of the open subset $V=\C^*\times \C^2\setminus\{0\}$ of $\C^3$ by the linear diagonal action of $\mathbb{Z}^2$ given by the representation
\begin{center}
    $
    \begin{matrix}
        \mathbb{Z}^2 & \longrightarrow & \mathrm{Diag}(\mathbb{C}^{n-3})^{\times} \\
        (1,0) & \longmapsto & A := 
        \begin{pmatrix}
            \alpha_1 & & \\
            & \alpha_2 & \\
            &&\alpha_3
        \end{pmatrix}\\
        (0,1) & \longmapsto & B := \begin{pmatrix}
            \beta_1 & & \\
            & \beta_2 & \\
            &&\beta_3
        \end{pmatrix} ,
    \end{matrix}
    $
\end{center}
where 
$$
\left\{
\begin{matrix}
    \alpha_j & = & \exp\left(2i\pi\left< \Lambda_{j+3}-\Lambda_1,\Omega^{-1}\begin{pmatrix}
        1 \\ 0
    \end{pmatrix} \right>\right) \\
    \beta_j & = & \exp\left(2i\pi\left< \Lambda_{j+3}-\Lambda_1,\Omega^{-1}\begin{pmatrix}
        0 \\ 1
    \end{pmatrix} \right>\right),
\end{matrix}
\right.
$$
and $\Omega$ is the $2\times 2$ matrix such that the rows are given by the vectors $\Lambda_2-\Lambda_1$ and $\Lambda_3-\Lambda_1$.

We use Douady-type sequences, but we first consider the quotient only by the group generated by the matrix $A$ and then, we take the quotient by the group generated by the second matrix $B$. We denote by $N_A$ the quotient of $V$ by the group generated by $A$ and $\pi_A:V\to N_A$ the quotient map. We have another quotient map, corresponding to the quotient of $N_A$ by the action of $B$, which we denote $\pi_B:N_A\to N$. Those three maps are holomorphic coverings and we have $\pi=\pi_B\circ\pi_A$.

We consider $\mathcal{U}=\{U_m\}$ an open covering of $N$ with Stein open subsets trivializing the covering map $\pi$. We observe that this covering also trivializes $\pi_B$ and that the covering $\mathcal{U}_A=\{\pi_B^{-1}(U_m)\}$ of Stein open sets of $N_A$ trivializes $\pi_A$. Let $\mathcal{U}_{A,B}=\{\pi^{-1}(U_m)\}$. The symbol $\Theta_X$ will stand for the sheaf of holomorphic vector field over a given complex manifold X. When the complex manifold $X$ is obvious, the index of the symbol $\Theta_X$ will be intentionally omitted, as in the following cochain spaces or in cohomology groups.

We have the following short exact sequences of cochain spaces

\begin{center}
    
    \begin{tikzcd}
        0 \arrow[r] &   \mathcal{C}^{\bullet}(\mathcal{U}_A,\Theta) \arrow[r,"\pi_A^*"]   &   \mathcal{C}^{\bullet}(\mathcal{U}_{A,B},\Theta) \arrow[r,"Id-A_*"]  &   \mathcal{C}^{\bullet}(\mathcal{U}_{A,B},\Theta) \arrow[r]  &    0
    \end{tikzcd} 
    
    \begin{tikzcd}
        0 \arrow[r] &   \mathcal{C}^{\bullet}(\mathcal{U},\Theta) \arrow[r,"\pi_B^*"]   &   \mathcal{C}^{\bullet}(\mathcal{U}_{A},\Theta) \arrow[r,"Id-B_*"]  &   \mathcal{C}^{\bullet}(\mathcal{U}_{A},\Theta) \arrow[r]  &    0.
    \end{tikzcd}

\end{center}

These two sequences induce two long exact sequences in cohomology that allow to compute the cohomology of $N$ with value in its tangent bundle. Namely, we get the following (non-canonical) isomorphisms
\small
\begin{center}
    $
    \begin{matrix}
        H^0(N_A,\Theta)\cong \ker ((Id-A_*)\vert_{ H^0(V,\Theta)}), \\ \\
        
        H^q(N_A,\Theta)\cong \coker ((Id-A_*)\vert_{ H^{q-1}(V,\Theta)})\oplus \ker ((Id-A_*)\vert_{ H^q(V,\Theta)}), ~\text{for} ~q\geq 1, \\ \\
    
        H^0(N,\Theta)\cong \ker ((Id-B_*)\vert_{ H^0(N_A,\Theta)}),\\ \\
    
        H^q(N,\Theta)\cong \coker ((Id-B_*)\vert_{ H^{q-1}(N_A,\Theta)})\oplus \ker ((Id-B_*)\vert_{ H^q(N_A,\Theta)}), ~\text{for} ~q\geq 1.
    \end{matrix}
    $
\end{center}
\normalsize

We have \textit{a priori} no control on these spaces any time there is a cokernel involved. But we have a strong tool, namely, the compactness of the LVM manifold, which ensures that each cohomology space of $N$ with value in $\Theta$ is finite dimensional.

By the above isomorphisms, we can express any element of any of these cohomology groups with vector fields defined in some open subsets of $\mathbb{C}^{3}$.

Using the Künneth formula, we can compute (\cite{LN96}, p.798) the cohomology $H^{\bullet}(V,\Theta)$ to get:

\begin{equation*}
    H^0(V,\Theta)\cong \left \{ \sum_{j=1}^3 \sum_{p\in \mathbb{Z}\times\mathbb{N}\times\mathbb{N}} a_{j,p} z^p \frac{\partial}{\partial z_j}~\text{convergent on } \mathbb{C}^*\times \mathbb{C}^2\right\},
\end{equation*}
\begin{equation*}
    H^1(V,\Theta)\cong \left \{ \sum_{j=1}^3 \sum_{p\in \mathbb{Z}\times -(\mathbb{N}^*)\times -(\mathbb{N}^*)} a_{j,p} z^p \frac{\partial}{\partial z_j}~\text{convergent on } (\mathbb{C}^*)^3\right\},
\end{equation*}
\begin{equation*}
    H^q(V,\Theta)=0,~ \text{for } q\geq 2.
\end{equation*}

\begin{Definition}\label{Definition Set of common resonances}
    For $j\in\{1,2,3\}$, $E~\subset \mathbb{Z}^{3}$ and $\mathcal{F}=\{\alpha=(\alpha_1,\alpha_2,\alpha_{3})\}$ a family of sets of eigenvalues, we denote by $\mathcal{R}_{j,E}^{\mathcal{F}}$ the set of common $\mathcal{F}$ multiplicative $j$-\textbf{resonances} in $E$
    \begin{center}
        $
        \mathcal{R}_{j,E}^{\mathcal{F}}=\left\{p\in E,~\forall\alpha =(\alpha_1,\alpha_2,\alpha_{3}) \in\mathcal{F},~\alpha_j=\alpha^p\right\}
        $
    \end{center}
    and by $VF_{E}^{\mathcal{F}}$ the set of common $\mathcal{F}$ multiplicative resonant vector fields with exponents in $E$
    \begin{center}
        $
        VF_{E}^{\mathcal{F}}=\left\{ \sum\limits_{j=1}^{3}\sum\limits_{p\in \mathcal{R}_{j,E}^{\mathcal{F}}}a_{j,p}z^p\frac{\partial}{\partial z_j},~\text{convergent on }(\mathbb{C}^*)^{3}\right\}.
        $
    \end{center}
\end{Definition}

\begin{Definition}\label{Definition Resonances for an LVM manifold}
    We will say that $(j,p)$ is a \textbf{resonance} of the LVM manifold $N=V/<A,B>$ if $p$ is an element of $\mathcal{R}_{j,\Z\times \N \times \N}^{\alpha,\beta}$, i.e. if $\alpha_j=\alpha^p$ and $\beta_j=\beta^p$. 
\end{Definition}

\begin{Remark}\label{Remark Definition of Non resonant and affinely resonant}
    As the reader can see, for $j\in\{1,2,3\}$, $(j,e_j)$, where $e_j$ is the $j$-th vector of the standard basis of $\Z^3$, is always a \textbf{trivial resonance} in our definition, since it corresponds to the relations $\alpha_j=\alpha_j$ and $\beta_j=\beta_j$. However, we will say that $N$ is \textbf{non-resonant} if it carries only the trivial relations.

    We will say that $N$ is \textbf{affinely resonant} if $N$ carries only resonances that correspond to equalities of eigenvalues, e.g. $\alpha_3=\alpha_2$ and $\beta_3=\beta_2$.

    The classification of all the possible resonances for an LVM manifold of type $(2,6,4)$ is achieved in Lemma \ref{Lemma Classification resonances 264}.
\end{Remark}

As we saw before, the cohomology spaces are built with kernels and cokernels. In the case of Hopf manifolds, we can easily check that each cokernel is always isomorphic to the corresponding kernel. All kernels are spaces of resonant vector fields, so in the case of Hopf manifolds, we can easily describe all the cohomology spaces. We can also show that those spaces are finite-dimensional (these are classical properties of linear maps of Poincaré type). 

Here, it seems to be more tricky. The kernels are still understandable as spaces of resonant vector fields, but for the cokernels, we have to build convergent series, and there is no apparent reason for these series to be convergent. As we said before, we have the strong tool of compactness.

Reading the cohomology spaces of the open set $V\subset \mathbb{C}^3$, we infer that the sets of common resonances appearing are $\mathcal{R}_{j,\mathbb{Z}\times\mathbb{N}\times\mathbb{N}}^{\{\alpha,\beta\}}$ and $\mathcal{R}_{j,\mathbb{Z}\times -(\mathbb{N}^*)\times -(\mathbb{N}^*)}^{\{\alpha,\beta\}}$ for $j\in \{1,2,3\}$ and with $\alpha=(\alpha_1,\alpha_2,\alpha_{3})$, respectively $\beta=(\beta_1,\beta_2,\beta_{3})$, the set of eigenvalues of the matrix $A$, respectively $B$. 

Each resonance of those spaces gives rise to an element of some cohomology space of the LVM manifold $N$, represented by a monomial vector field $z^p\frac{\partial}{\partial z_j}$. Therefore, every family of resonances produces a linearly independent family of cohomology classes. By compactness, the set of resonances is therefore of finite cardinality. Using this argument, we can prove the following lemma.

\begin{Lemma}\label{Lemma Impossible resonances in the case LVM 264}
    \begin{enumerate}[label=(\roman*)]
        \item For every $j\in \{1,2,3\}$, the set of common resonances 
        $$
        \mathcal{R}_{j,\mathbb{Z}\times -(\mathbb{N}^*)\times -(\mathbb{N}^*)}^{\{\alpha,\beta\}}
        $$ is empty,
        \item the set $\mathcal{R}_{1,\mathbb{Z}\times\mathbb{N}\times\mathbb{N}}^{\{\alpha,\beta\}}$ contains only the trivial resonance $(1,0,0)$,
        \item for every $j\in \{2,3\}$ and $p\in \mathcal{R}_{j,\mathbb{Z}\times\mathbb{N}\times\mathbb{N}}^{\{\alpha,\beta\}}$, we have $p_j=0$ or $p=e_j$, where $(e_1,e_2,e_3)$ is the standard basis of $\mathbb{Z}^3$.
    \end{enumerate}
\end{Lemma}

\begin{proof}\label{Preuve Lemma Impossible resonances in the case LVM 264}
    The argument is essentially the same as considering the eigenvalues of the linear part of a vector field in the Poincaré domain when we want to find a normal form of a vector field in the neighborhood of a singularity. If we have a relation $\lambda_1+\lambda_2=0$, we can build infinitely many resonance relations, such as $\lambda_1=\lambda_1+l(\lambda_1+\lambda_2)$ for $l\in\mathbb{N}$. 

    \textit{(i)} Let $j$ be an element of $\{1,2,3\}$ and assume that there exists an element $p$ in the set $\mathcal{R}_{j,\mathbb{Z}\times -(\mathbb{N}^*)\times -(\mathbb{N}^*)}^{\{\alpha,\beta\}}$. This means that we have 

    \begin{center}
        $
        \left\{
        \begin{matrix}
            \alpha_j=\alpha^p   \\
            \beta_j=\beta^p,
        \end{matrix}
        \right.
        $
    \end{center}    

    which we can rewrite

    \begin{center}
        $
        \left\{
        \begin{matrix}
            1=\alpha_j^{-1}\alpha^p   \\
            1=\beta_j^{-1}\beta^p.
        \end{matrix}
        \right.
        $
    \end{center}

    Using both systems, we infer that for every $l\in \mathbb{N}$, we have 

    \begin{center}
        $
        \left\{
        \begin{matrix}
            \alpha_j=\alpha^{(l+1)p-le_j}   \\
            \beta_j=\beta^{(l+1)p-le_j}.
        \end{matrix}
        \right.
        $
    \end{center} 

    But for every $l\in \mathbb{N}$, $(l+1)p-le_j$ is an element of $\mathbb{Z}\times -(\mathbb{N}^*)\times -(\mathbb{N}^*)$ and for $l\ne l'$, we have $(l+1)p-le_j\ne (l'+1)p-l'e_j$, since $p\ne e_j$. Therefore, the set $\mathcal{R}_{j,\mathbb{Z}\times -(\mathbb{N}^*)\times -(\mathbb{N}^*)}^{\{\alpha,\beta\}}$ is infinite, which is impossible. 

    The proofs of \textit{(i)} and \textit{(ii)} are achieved the same way.
\end{proof}

\textit{A priori}, our cohomology spaces may contain more than just resonant vector fields. A non trivial and non resonant element of a cokernel would be a power series. But we will use the compactness of $N$ to prove that such a power series cannot represent an element in a finite dimensional cohomology group.

\begin{Proposition}\label{Proposition Cohomology in Theta for a LVM manifold of type 264}
    We have the following results:
    \begin{center}
    $
    \begin{matrix}
        H^0(N,\Theta)\cong VF^{\alpha,\beta}_{\mathbb Z \times \mathbb N \times \mathbb N}, \\
        H^1(N,\Theta)\cong \left(VF^{\alpha,\beta}_{\mathbb Z \times \mathbb N \times \mathbb N}\right)^2, \\
        H^2(N,\Theta)\cong VF^{\alpha,\beta}_{\mathbb Z \times \mathbb N \times \mathbb N}, \\
        H^q(N,\Theta)=\{0\} \text{ for } q\geq 3.
    \end{matrix}
    $
    \end{center}
\end{Proposition}

\begin{proof}\label{Preuve Proposition Cohomology in Theta for a LVM manifold of type 264}
    The computation of $H^0(N,\Theta)$ is easy and the result is the expected one. The global vector fields on $N$ are exactly the $<A,B>$-invariant global vector fields on $V$, namely the elements of $VF^{\alpha,\beta}_{\mathbb Z \times \mathbb N \times \mathbb N}$.

    As we saw before, we have 
    $$
    H^1(N,\Theta)\cong \coker ((Id-B_*)\vert_{ H^{0}(N_A,\Theta)})\oplus \ker ((Id-B_*)\vert_{ H^1(N_A,\Theta)}).
    $$
    
    Let $x\in\coker((Id-B^*)\vert_{H^0(N_A,\Theta)})$. Since $\pi_A^*$ is a $<B>$-invariant isomorphism between $H^0(N_A,\Theta)$ and $\ker((Id-A_*)\vert_{H^0(V,\Theta)})$, $x$ is represented by a power series
    $$
    \widetilde{x}=\sum\limits_{j=1}^3\sum\limits_{p\in\mathcal{R}^\alpha_{j,\mathbb Z\times\mathbb N\times\mathbb N}}a_{j,p}z^p\frac{\partial}{\partial z_j}.
    $$
    
    Let $\widetilde{x_\beta}$ be the part of $\widetilde{x}$ containing $<B>$-resonant monomials. We have $\widetilde{x_\beta}\in VF^{\alpha,\beta}_{\mathbb Z \times \mathbb N \times \mathbb N}$ so it is a finite sum.

    Now, if we assume that $[\widetilde{x}-\widetilde{x_\beta}]$ is a nonzero element of the space 
    $$
    \coker((Id-B_*)\vert_{\ker((Id-A_*)\vert_{H^0(V,\Theta)})}).
    $$
    It means that $\widetilde{x}-\widetilde{x_\beta}$ is a convergent power series such that, when we multiply the general term with 
    $$
    \frac{1}{1-\beta_j\beta^p},
    $$
    the series does not converge anymore. But this property is also verified by all the derivatives of the series $\widetilde{x}-\widetilde{x_\beta}$, and all these series produce infinite independent families of element in $H^1(N,\Theta)$, which is impossible. 
    
    Therefore, $[\widetilde{x}-\widetilde{x_\beta}]=0$, which proves that 
    $$
    \coker((Id-B^*)\vert_{H^0(N_A,\Theta)})\cong VF^{\alpha,\beta}_{\mathbb Z \times \mathbb N \times \mathbb N}.
    $$

    Let $x\in\ker ((Id-B_*)\vert_{ H^1(N_A,\Theta)})\subset H^1(N_A,\Theta)$. Since we have the equality 
    $$
    (Id-B_*)\circ \pi_A^*=\pi_A^*\circ (Id-B_*),
    $$
    we have that $\pi_A^*(x)$ is an element of 
    $$
    \ker((Id-A_*)\vert_{H^1(V,\Theta)})\cap \ker((Id-B_*)\vert_{H^1(V,\Theta)}).
    $$
    But this space is trivial by Lemma \ref{Lemma Impossible resonances in the case LVM 264} \textit{(i)} since it is isomorphic to 
    $$
    VF^{\alpha,\beta}_{\mathbb Z \times -(\mathbb N^*) \times -(\mathbb N^*)}.
    $$

    Therefore, $x$ is an element of $\im(\sigma_1)$, where $\sigma_1:H^0(V,\Theta)\to H^1(N_A,\Theta)$ is the map given by the Douady long exact sequence. Let $y\in H^0(V,\Theta)$ be such that $x=\sigma_1(y)$. We would like to show that $[y]$ in $\coker((Id-A_*)\vert_{H^0(V,\Theta)})$ is an element of $VF^{\alpha,\beta}_{\mathbb Z \times \mathbb N \times \mathbb N}$.

    Since we have the equality 
    $$
    (Id-B_*)\circ \sigma_1=\sigma_1\circ (Id-B_*),
    $$
    we have that 
    $$
    \begin{matrix}
        \sigma_1((Id-B_*)(y)) & = & (Id-B_*)(\sigma_1(y)) \\
                              & = & (Id-B_*)(x) \\
                              & = & 0.
    \end{matrix}        
    $$
    
    Therefore, $(Id-B_*)(y)\in\ker(\sigma_1)=\im((Id-A_*)\vert_{H^0(V,\Theta)})$. But since $A$ and $B$ are linear diagonal, the monomials of the power series $y$ that are $<A>$-invariant should be also $<B>$-invariant, and for the same reason as before, the difference between $y$ and these monomials in $VF^{\alpha,\beta}_{\mathbb Z \times \mathbb N \times \mathbb N}$ is zero in $\coker((Id-A_*)\vert_{H^0(V,\Theta)})$.

    We just proved that 
    $$
    H^1(N,\Theta)\cong \left(VF^{\alpha,\beta}_{\mathbb Z \times \mathbb N \times \mathbb N}\right)^2.
    $$

    To finish the proof, we need to do the same thing for $H^2$ and $H^3$ (since $N$ is of complex dimension $3$). The arguments are the same, we follow each element in our long exact sequences, we use commutation of all the maps we need, we use Lemma \ref{Lemma Impossible resonances in the case LVM 264} to prove that some kernels are trivial of finite dimensional and we use the compactness of $N$ to prove that no power series can represent an element of a cohomology group in a cokernel.
\end{proof}

\subsection{Kuranishi family of a non-resonant (2,6,4) LVM manifold}\label{Subsection Kuranishi family of a non resonant (2,6,4) LVM manifold}

Let $\Lambda^{(0)}=(\Lambda^{(0)}_1,...,\Lambda^{(0)}_6)$ be an LVM configuration of type $(2,6,4)$ that generates an LVM manifold $N_0=V / <A_0,B_0>$ with $V=\mathbb{C}^*\times (\mathbb{C}^2\setminus \{0\})$, and

\begin{center}
    $
    A_0=\begin{pmatrix}
        \alpha^{(0)}_1  &   0   &   0   \\
        0   &   \alpha^{(0)}_2  &   0   \\
        0   &   0   &   \alpha^{(0)}_3
    \end{pmatrix},
    $
\end{center}

\begin{center}
    $
    B_0=\begin{pmatrix}
        \beta^{(0)}_1  &   0   &   0   \\
        0   &   \beta^{(0)}_2  &   0   \\
        0   &   0   &   \beta^{(0)}_3
    \end{pmatrix},
    $
\end{center}

\begin{center}
    $
    \alpha^{(0)}_j=\exp\left(2i\pi\left<\Lambda^{(0)}_{j+3}-\Lambda^{(0)}_1,\Omega^{-1}\begin{pmatrix}
        1   \\
        0
    \end{pmatrix}\right>\right)
    $
    for $j\in\{1,2,3\}$,
\end{center}

\begin{center}
    $
    \beta^{(0)}_j=\exp\left(2i\pi\left<\Lambda^{(0)}_{j+3}-\Lambda^{(0)}_1,\Omega^{-1}\begin{pmatrix}
        0   \\
        1
    \end{pmatrix}\right>\right)
    $
    for $j\in\{1,2,3\}$,
\end{center}
where $\Omega$ is the $2\times 2$ invertible matrix with rows given by $\Lambda^{(0)}_{2}-\Lambda^{(0)}_1$ and $\Lambda^{(0)}_{3}-\Lambda^{(0)}_1$.

Following Lemma \ref{Lemma Deformation space in the case of k plus grand que m+1}, we consider the deformation space 
\small
$$
    D=\left\{
    \begin{matrix}
        \Lambda=(\Lambda^{(0)}_1,\Lambda^{(0)}_2,\Lambda^{(0)}_3,\Lambda_4,\Lambda_5,\Lambda_6) \text{ LVM configuration equivalent to }    \Lambda^{(0)}
    \end{matrix}
    \right\}.
$$
\normalsize
Every element $\Lambda\in D$ generates two diagonal matrices $A_\Lambda$ and $B_\Lambda$, defined as $A_0$ and $B_0$, and an LVM manifold $N_\Lambda=V / <A_\Lambda,B_\Lambda>$.

We consider the two biholomorphic maps 

\begin{center}
    $
    \begin{matrix}
        g_1     &   :   &   V\times D   &   \longrightarrow     &   V\times D   \\
                &       &   (\xi,\Lambda)   &   \longmapsto     &   (A_\Lambda\xi,\Lambda),  \\
        g_2     &   :   &   V\times D   &   \longrightarrow     &   V\times D   \\
                &       &   (\xi,\Lambda)   &   \longmapsto     &   (B_\Lambda\xi,\Lambda).
    \end{matrix}
    $
\end{center}

Obviously, these two maps commute and generate a fixed point free and proper action of $\mathbb{Z}^2$ on $V\times D$.

Therefore, we build a family of deformations of the LVM manifold $N_0$ considering the quotient $\mathcal{M}=(V\times D) / <g_1,g_2>$ and the natural projection 

\begin{center}
    $
    \begin{matrix}
        \varpi  &   :   &   \mathcal{M}  &   \longrightarrow     &   D   \\
                &       &   [\xi,\Lambda]   &   \longmapsto      &   \Lambda,
    \end{matrix}
    $
\end{center}
whose fibers are $\varpi^{-1}(\Lambda)=N_\Lambda$.

The aim of this section is to describe the Kodaira-Spencer map of this family. The calculations are similar to the ones of Wehler for Hopf surfaces (\cite{W81}). From now on, we drop the zeros in the indices and exponents for $N$, $A$, $B$, $\alpha$ and $\beta$. We will use the following notation.
\begin{center}
    \begin{tikzcd}[row sep= small]
         & N_A \arrow[rd, "\pi_B"] & \\
        V \arrow[ru, "\pi_A"] \arrow[rr, "\pi"] \arrow[rd, "p_B"] & & N, \\
         & N_B \arrow[ru, "p_A"] & 
    \end{tikzcd}
\end{center}

\begin{center}
    \begin{tikzcd}[column sep=small]
        0 \arrow[r] & H^0(N_A) \arrow[r,"\pi_A^*"] & H^0(V) \arrow[r, "Id-A_*"] & H^0(V) \arrow[r, "\sigma_1"] & H^1(N_A) \arrow[r, "\pi_A^*"] & H^1(V) \arrow[r] & ..., \\
        0 \arrow[r] & H^0(N) \arrow[r,"\pi_B^*"] & H^0(N_A) \arrow[r, "Id-B_*"] & H^0(N_A) \arrow[r, "\sigma_2"] & H^1(N) \arrow[r, "\pi_B^*"] & H^1(N_A) \arrow[r] & ..., \\
        0 \arrow[r] & H^0(N_B) \arrow[r,"p_B^*"] & H^0(V) \arrow[r, "Id-B_*"] & H^0(V) \arrow[r, "\tau_1"] & H^1(N_B) \arrow[r, "p_B^*"] & H^1(V) \arrow[r] & ..., \\
        0 \arrow[r] & H^0(N) \arrow[r,"p_A^*"] & H^0(N_B) \arrow[r, "Id-A_*"] & H^0(N_B) \arrow[r, "\tau_2"] & H^1(N) \arrow[r, "p_A^*"] & H^1(N_B) \arrow[r] & ..., \\
    \end{tikzcd}
    $
    \begin{matrix}
        g_{1,1}     &   :   &   V\times D   &   \longrightarrow     &   V   \\
                    &       &   (\xi,\Lambda)   &   \longmapsto     &   A_\Lambda\xi
    \end{matrix}
    $
\end{center}

\begin{center}
    $
    \begin{matrix}
        g_{2,1}     &   :   &   V\times D   &   \longrightarrow     &   V   \\
                    &       &   (\xi,\Lambda)   &   \longmapsto     &   B_\Lambda\xi.
    \end{matrix}
    $
\end{center}

\begin{Lemma}\label{Lemma Kodaira-Spencer map of the LVM family}
    The Kodaira-Spencer map $KS$ of the family $\varpi$ is given by 
    $$
    \begin{matrix}
        T_{\Lambda^{(0)}}D & \longrightarrow & H^1(N,\Theta) \\
        v & \longmapsto & \tau_2(p_{B*}(\frac{\partial g_{1,1}}{\partial\Lambda}(A^{-1}\bullet~,\Lambda^{(0)})v))+\sigma_2(\pi_{A*}(\frac{\partial g_{2,1}}{\partial\Lambda}(B^{-1}\bullet~,\Lambda^{(0)})v)).
    \end{matrix}
    $$
\end{Lemma}

\begin{proof}
    After choosing a covering of $N$ by Stein open subsets that trivialize the map $\pi:V\to N$, we use this open subsets to trivialize $\varpi$. We lift locally the vector $v\in T_{\Lambda^{(0)}}D$ and we construct a \v Cech cocycle representing the class $KS(v)$. In order to prove that this class corresponds to the one given in the Lemma, one need to prove that the vector field $\frac{\partial g_{1,1}}{\partial\Lambda}(A^{-1}\bullet~,\Lambda^{(0)})v$ is $B$-invariant, the vector field $\frac{\partial g_{2,1}}{\partial\Lambda}(B^{-1}\bullet~,\Lambda^{(0)})v$ is $A$-invariant and that the space $\im(\sigma_2)\cap\im(\tau_2)$ is trivial.
\end{proof}

We deduce the following result.

\begin{Theorem}\label{Theorem Kuranishi family for LVM manifolds of type (2,6,4) non resonant}
    Let $(\Lambda_1^{(0)},...,\Lambda_6^{(0)})$ be an LVM configuration of type $(2,6,4)$. Assume that the configuration is non-resonant. Then the Kodaira-Spencer map of the LVM family $\varpi$ at $(\Lambda_1^{(0)},...,\Lambda_6^{(0)})$ is an isomorphism. Moreover, this family is the Kuranishi family of the LVM manifold generated by this configuration.
\end{Theorem}

\begin{Remark}\label{Remark Being non resonant is generic}
    The condition of being non-resonant is a generic condition. We emphasize that the Kuranishi space of such a non-resonant LVM manifold of type $(2,6,4)$ is smooth, and that a small deformation of it is still a non-resonant LVM  manifold of type $(2,6,4)$.
\end{Remark}

\begin{proof}\label{Preuve Theorem Kuranishi family for LVM manifolds of type (2,6,4) non resonant}
    Since the deformation space $D$ is smooth, according to Corollary \ref{Corollary Smooth family with invertible KS map}, we only need to show that the Kodaira-Spencer map is an isomorphism. This is a direct consequence of Lemma \ref{Lemma Kodaira-Spencer map of the LVM family} and \ref{Proposition Cohomology in Theta for a LVM manifold of type 264}.
\end{proof}

\subsection{First-order obstructions to deformation}\label{Subsection First order obstructions to deformation}

Let $(\Lambda_1,...,\Lambda_6)$ be an LVM configuration of type $(2,6,4)$ that generates an LVM manifold $N=V / <A,B>$ with the same notations as usual.

Following the classical theory of Kodaira, Spencer, and Kuranishi, the cohomology space $H^1(N,\Theta)$ is isomorphic to the space of infinitesimal deformations. The first natural question is whether an element of $H^1(N,\Theta)$ can be in the image of the Kodaira-Spencer map of a family having $(\mathbb C,0)$ as base space. This question is solved by the theory of obstructions, as we mentioned in Section \ref{Subsection Deformations of the complex structure}. 

In terms of \v{C}ech cohomology, if $\theta\in H^1(N,\Theta)$ is represented by a $1$-cocycle $\{\theta_{i,j}\}$, the element $[\theta,\theta]\in H^2(N,\Theta)$ (i.e. the first obstruction) is represented by the $2$-cocycle $\{[\theta_{i,j},\theta_{j,k}]\}$. Notice that in the last cocycle, $[.,.]$ is the classical Lie bracket of vector fields.

We take back the notations of Sections \ref{Subsection  Cohomology in the holomorphic tangent sheaf} and \ref{Subsection Kuranishi family of a non resonant (2,6,4) LVM manifold}, that we will recall briefly. We have two exact sequences in cohomology 
\begin{center}
    \begin{tikzcd}[column sep = small]
        0 \arrow[r]     &   H^0(N_A) \arrow[r, "\pi_A^*"]    &   H^0(V) \arrow[r, "Id-A_*"]     &   H^0(V) \arrow[r, "\sigma_1"]  &   H^1(N_A) \arrow[r, "\pi_A^*"]     &   H^1(V) \arrow[r, "Id-A_*"]    &   ...,
    \end{tikzcd}
\end{center}
\begin{center}
    \begin{tikzcd}[column sep = small]
        0 \arrow[r]     &   H^0(N) \arrow[r, "\pi_B^*"]    &   H^0(N_A) \arrow[r, "Id-B_*"]     &   H^0(N_A) \arrow[r, "\sigma_2"]  &   H^1(N) \arrow[r, "\pi_B^*"]     &   H^1(N_A) \arrow[r, "Id-B_*"]    &   ...~.
    \end{tikzcd}
\end{center}

We recall that, according to Proposition \ref{Proposition Cohomology in Theta for a LVM manifold of type 264}, one has isomorphisms 
$$
\begin{matrix}
    \left(VF^{\alpha,\beta}_{\mathbb Z \times \mathbb N \times \mathbb N}\right)^2 & \longrightarrow & H^1(N,\Theta)\\
    (X,Y) & \longmapsto & \sigma_2((\pi_A) _* X)+(\pi_B) _*(\sigma_1(Y))
\end{matrix}
$$
and 
$$
\begin{matrix}
    VF^{\alpha,\beta}_{\mathbb Z \times \mathbb N \times \mathbb N} & \longrightarrow & H^2(N,\Theta) \\
    Z & \longmapsto & c_2(\sigma_1(Z)),
\end{matrix}
$$
where $c_2: H^1(N_A,\Theta) \to H^2(N,\Theta) $ is the connecting morphism of the second long exact sequence in cohomology.

\begin{Proposition}\label{Proposition First obstruction}
    The first obstruction map to deformation reads, through these isomorphisms, like 
    $$
    \begin{matrix}
        \left(VF^{\alpha,\beta}_{\mathbb Z \times \mathbb N \times \mathbb N}\right)^2 & \longrightarrow & VF^{\alpha,\beta}_{\mathbb Z \times \mathbb N \times \mathbb N} \\
        (X,Y) & \longmapsto & [X,Y].
    \end{matrix}
    $$
\end{Proposition}

\begin{proof}
    This result follows from using the isomrphisms given just above and use the \v Cech cohomology to compute the bracket following the Douady long exact sequences.
\end{proof}

\section{Resonant structures}\label{Section Resonant structures}

\subsection{Group of resonant transformations}\label{Subsection Group of resonant transformations}

Let $(\Lambda_1,...,\Lambda_6)$ be an LVM configuration of type $(2,6,4)$ that generates an LVM manifold $N=V / <A,B>$ with the same notations as usual. The following Lemma is a direct consequence of Lemma \ref{Lemma Impossible resonances in the case LVM 264}.

\begin{Lemma}\label{Lemma Classification resonances 264}
    The matrices $A$ and $B$ carry at most two common non-trivial resonances which (exchanging $\Lambda_5$ and $\Lambda_6$ if needed) are generated by relations
    $$
    \left\{
    \begin{matrix}
        \alpha_3=(\alpha_1)^p(\alpha_2)^q \\
        \beta_3=(\beta_1)^p(\beta_2)^q,
    \end{matrix}
    \right.
    $$
    where $p\in\mathbb Z$ and $q\in\mathbb N^*$. Moreover, the case with exactly two common non-trivial resonances occurs when $q=1$.
\end{Lemma}

Let $G$ be the Lie group of invertible transformations of $V=\mathbb C^* \times (\mathbb C^2\setminus \{0\})$ of the form
$$
z\longmapsto \sum\limits_{j=1}^3\sum\limits_{p\in\mathcal{R}_{j,\mathbb{Z}\times\mathbb{N}\times\mathbb{N}}^{\{\alpha,\beta\}}}a_{j,p}z^pe_j,
$$
where $e_j$ is the $j$-th vector of the standard basis of $\C^3$. Notice that it corresponds to the connected component of the identity of the group of automorphisms of $N$. 

The next proposition gives more precision on the Lie group $G$, depending on the resonances.

\begin{Proposition}\label{Proposition Groupe des transfo résonantes}
    \begin{enumerate}[label=(\roman*)]
        \item If the matrices $A$ and $B$ do not carry any non-trivial resonance, then $G$ is the Lie group of invertible diagonal linear transformations of $\mathbb C^3$ and therefore, it is an abelian Lie group of dimension $3$. 
        \item If the matrices $A$ and $B$ carry exactly one common non-trivial resonance, then $G$ is of dimension $4$.
        \item If the matrices $A$ and $B$ carry exactly two common non trivial resonance, then $G$ is of dimension $5$. Moreover, $G$ has a structure of semi-direct product $\mathbb C^* \rtimes \mathrm{GL}(2,\mathbb C)$ given by the homomorphism 
        $$
        \begin{matrix}
            \tau_p & : & \mathbb C^* & \longrightarrow & \Aut(\mathrm{GL}(2,\mathbb C)) \\
                 &   & z      & \longmapsto     & \tau_p(z),
        \end{matrix}
        $$
        where
        $$
        \tau_p(z)
        \left(
        \begin{pmatrix}
            a & b \\
            c & d
        \end{pmatrix}
        \right)
        =
        \begin{pmatrix}
            a & z^{-p}b \\
            z^pc & d
        \end{pmatrix},
        $$
        where $p\in\mathbb Z$ is the integer such that 
        $$
        \left\{
        \begin{matrix}
            \alpha_3=(\alpha_1)^p\alpha_2 \\
            \beta_3=(\beta_1)^p\beta_2.
        \end{matrix}
        \right.
        $$
        Finally, any element $(\alpha,A)$ of $G$ is in the conjugacy class of an element of the form $(\alpha,T)$ where $T$ is a inferior triangular matrix. An element of the latter form is called \textbf{triangular}.
        If these resonances are affine, then $G$ is the direct product $\mathbb C^*\times \mathrm{GL}(2,\mathbb C)$, as we saw before, and the reduction part is classical. 
    \end{enumerate}
\end{Proposition}

\begin{proof}\label{Preuve Proposition Groupe des transfo résonantes}
    (i) and (ii) are straightforward.

    We prove (iii). We consider two general elements of $G$
    $$
    \begin{matrix}
        f & : & V & \longrightarrow & V \\
          &   & (z_1,z_2,z_3) & \longmapsto & \left(\alpha z_1,\tau_p(z_1)
          \left(
          \begin{pmatrix}
              a & b \\
              c & d
          \end{pmatrix}
          \right)
          \begin{pmatrix}
              z_2 \\
              z_3
          \end{pmatrix}
          \right) \\
          g &  : & (z_1,z_2,z_3) & \longmapsto & \left(\alpha' z_1,\tau_p(z_1)
          \left(
          \begin{pmatrix}
              a' & b' \\
              c' & d'
          \end{pmatrix}
          \right)
          \begin{pmatrix}
              z_2 \\
              z_3
          \end{pmatrix}
          \right),
    \end{matrix}
    $$
    where $\alpha,\alpha'$ are elements of $\mathbb C^*$, 
    $
    \begin{pmatrix}
        a & b \\
        c & d
    \end{pmatrix}
    $
    and 
    $
    \begin{pmatrix}
        a' & b' \\
        c' & d'
    \end{pmatrix}
    $
    are invertible matrices, and $p$ is as in the statement. Then, a direct computation shows that 
    \small
    $$
    \begin{matrix}
        f\circ g & :  & (z_1,z_2,z_3) & \longmapsto & \left(\alpha\alpha' z_1,\tau_p(z_1)\left(\tau_p(\alpha')
          \left(
          \begin{pmatrix}
              a & b \\
              c & d
          \end{pmatrix}
          \right)
          \begin{pmatrix}
              a' & b' \\
              c' & d'
          \end{pmatrix}
          \right)
          \begin{pmatrix}
              z_2 \\
              z_3
          \end{pmatrix}
          \right)
    \end{matrix}
    $$
    \normalsize
    and proves the semi-direct product assertion. From now on, $f$ will be denoted as $(\alpha,A)$ with 
    $$
    A= 
    \begin{pmatrix}
        a & b \\
        c & d
    \end{pmatrix}.
    $$
    Let $P$ be an element of $\mathrm{GL}(2,\mathbb C)$ and let $\phi:=(1,P)$ be the element of $G$ defined by
    $$
    \begin{matrix}
        \phi & : & V & \longrightarrow & V \\
          &   & (z_1,z_2,z_3) & \longmapsto & \left( z_1,\tau_p(z_1)
          \left(
          P
          \right)
          \begin{pmatrix}
              z_2 \\
              z_3
          \end{pmatrix}
          \right).
    \end{matrix}
    $$
    Then, we have
    $$
    \begin{matrix}
        \phi^{-1}f\phi & = & (1,P^{-1})\rtimes(\alpha,A)\rtimes(1,P) \\
                       & = & (\alpha,\tau_p(\alpha)(P^{-1})A)\rtimes(1,P) \\
                       & = & (\alpha,\tau_p(\alpha)(P^{-1})AP).
    \end{matrix}    
    $$
    We consider the matrix $L_{\alpha,p}$ defined by 
    $$
    L_{\alpha,p}=
    \begin{pmatrix}
        1 & 0 \\
        0 & \alpha^p
    \end{pmatrix}.
    $$
    The purpose of this matrix is that 
    $$
    \tau_p(\alpha)(P^{-1})=L_{\alpha,p}P^{-1}L_{\alpha,p}^{-1}.
    $$
    The complex polynomial $\det(XL_{\alpha,p}-A)$ is of degree $2$, because $\alpha\neq 0$. Therefore, it has a complex root $\lambda$. Hence, the linear map $A-\lambda L_{\alpha,p}$ is not one-to-one, so there exists a nonzero vector $y\in\mathbb C^2$ such that $Ay=\lambda L_{\alpha,p}y$. Let $x$ be such that $(x,y)$ is a basis of $\mathbb C^2$ and let $P$ be the matrix of this basis in the canonical basis. Then $P$ is such that $\tau_p(\alpha)(P^{-1})AP$ is triangular.
\end{proof}

The Lie group $G$ is a group of analytic automorphisms of the manifold $V$. Therefore, we can consider $(G,V)$-structures and call such a structure a \textbf{resonant structure}.

The compact manifold $N$ carries a natural complete $(G,V)$-structure given by the couple $(\mathrm{Dev}_0,\rho_0)$ (see Section \ref{Subsection G,X structures}) defined as follows:
$$
\begin{matrix}
    \mathrm{Dev}_0 & : & \mathbb C \times (\mathbb C^2\setminus \{0\}) & \longrightarrow & V \\
          &   & (w_1,\xi_2,\xi_3)                             & \longmapsto     & (e^{2i\pi w_1},\xi_2,\xi_3) \\
    \rho_0 & : & \mathbb{Z}^3                                 & \longrightarrow & G \\
           &   & (p,q,r)                                      & \longmapsto     & A^pB^q  .
\end{matrix}
$$

\subsection{Deformations of the (G,V)-structure in the non-resonant case}\label{Subection Deformations in the non-resonant case}

Let $(\Lambda_1^{(0)},...,\Lambda_6^{(0)})$ be an LVM configuration of type $(2,6,4)$ that generates an LVM manifold $N_0=V / <A_0,B_0>$ with the same notations as usual. 

We assume in this section that $A_0$ and $B_0$ do not carry any non trivial common resonances. In that case, the group $G$ is the group $\mathrm{Diag}(\C^3)^\times$ of invertible complex linear diagonal endomorphisms of $\C^3$.

Therefore, we consider an element $\rho$ in $\Hom(\mathbb{Z}^{3},\mathrm{Diag}(\mathbb C^3)^\times)$ that is close enough to $(A_0,B_0,Id)$ and we denote 
\begin{center}
    $
    \begin{matrix}
        \rho:    &    \mathbb{Z}^{3}    &   \longrightarrow     &   \mathbb C ^*\times (\mathbb C^2\setminus \{0\})   \\
            &   (p,q,r)   &   \longmapsto     &   (0,A^pB^qC^r),
    \end{matrix}
    $
\end{center}
with 

\begin{center}
    $
    \begin{matrix}
        A=
        \begin{pmatrix}
            \alpha_1 & 0 & 0 \\
            0 & \alpha_2 & 0 \\
            0 & 0 & \alpha_3
        \end{pmatrix},
        \\
        B=
        \begin{pmatrix}
            \beta_1 & 0 & 0 \\
            0 & \beta_2 & 0 \\
            0 & 0 & \beta_3
        \end{pmatrix},
        \\
        C=
        \begin{pmatrix}
            \gamma_1 & 0 & 0 \\
            0 & \gamma_2 & 0 \\
            0 & 0 & \gamma_3
        \end{pmatrix},
    \end{matrix}
    $
\end{center}
where 

\begin{center}
    $
    \left\{
    \begin{matrix}
        \alpha_j \text{ is close to } \alpha^{(0)}_j, \\
        \beta_j \text{ is close to } \beta^{(0)}_j, \\
        \gamma_j \text{ is close to } 1 .
    \end{matrix}
    \right.
    $
\end{center}

Since $\gamma_j \text{ is close to } 1$ for $j\in \{1,2,3\}$, we consider $c_j=\frac{1}{2i\pi}\log(\gamma_j)$, where $\log$ is the principal branch of the complex logarithm.

In order to describe the $(G,V)$-structure with holonomy homomorphism $\rho$, we have to find the developing map. Another interesting question is which complex structure is induced by the $(G,V)$-structure. A naive answer would be that the complex structure could given by the quotient $V/ <A,B>$. In fact, in the case where $C$ is not the identity matrix, this is no longer true. However, we shall see that we already know the complex structure.

\begin{Theorem}\label{Theorem Developping map for the nearby affine structure in the non resonant case}
    The developing map of the $(G,V)$-structure with holonomy homomorphism $\rho$ is given by the map 
    \begin{center}
        $
        \begin{matrix}
            dev_{\rho} & : & \mathbb C \times (\mathbb C^2\setminus \{0\}) & \longrightarrow & V \\
                &   & (w_1,\xi_2,\xi_3) & \longmapsto & (e^{2i\pi w_1(1+c_1)},e^{2i\pi w_1c_2}\xi_2,e^{2i\pi w_1c_3}\xi_3).
        \end{matrix}
        $
    \end{center}
    More precisely, there exists a holomorphic submersion $\Psi:\rho \mapsto (\Lambda_4^\rho,\Lambda_5^\rho,\Lambda_6^\rho)$ from a small neighborhood of $\rho_0$ in $\Hom(\mathbb{Z}^{3},\mathrm{Diag}(\mathbb C^3)^\times)\cong (\mathbb C^*)^9$ to a small neighborhood of the element $(\Lambda_4^{(0)},\Lambda_5^{(0)},\Lambda_6^{(0)})$ in $(\mathbb C^2)^3$ such that $\Psi(\rho_0)=(\Lambda_4^{(0)},\Lambda_5^{(0)},\Lambda_6^{(0)})$ and such that $\Lambda(\rho):=(\Lambda_1^{(0)},\Lambda_2^{(0)},\Lambda_3^{(0)},\Psi(\rho))$ is an LVM configuration of type $(2,6,4)$ that generates an LVM manifold $N_\rho=V/<A_\rho,B_\rho>$ with
    
    \begin{center}
        $
        \begin{matrix}
            A_\rho=
            \begin{pmatrix}
                \alpha_1^\rho & 0 & 0 \\
                0 & \alpha_2^\rho & 0 \\
                0 & 0 & \alpha_3^\rho
            \end{pmatrix},
            \\
            B_\rho=
            \begin{pmatrix}
                \beta_1^\rho & 0 & 0 \\
                0 & \beta_2^\rho & 0 \\
                0 & 0 & \beta_3^\rho
            \end{pmatrix},
            \\
            \alpha_j^\rho=\exp\left(2i\pi\left<\Lambda_{j+3}^\rho-\Lambda_1^{(0)},\Omega^{-1}\begin{pmatrix}
                1 \\
                0
            \end{pmatrix}\right>\right) \text{ for }j\in\{1,2,3\}
            ,\\
            \beta_j^\rho=\exp\left(2i\pi\left<\Lambda_{j+3}^\rho-\Lambda_1^{(0)},\Omega^{-1}\begin{pmatrix}
                0 \\
                1
            \end{pmatrix}\right>\right) \text{ for }j\in\{1,2,3\}
        \end{matrix}
        $
    \end{center}
    and such that, if 

    \begin{center}
        $
        \rho(0,0,1)=\mathrm{Diag} (e^{2i\pi c_1},e^{2i\pi c_2},e^{2i\pi c_3})
        $
    \end{center}
    and $dev_\rho$ is as given just before, then, for every $(w_1,\xi_2,\xi_3)\in \mathbb C \times (\mathbb C^2\setminus \{0\})$, we have

    \begin{center}
        $
        \begin{matrix}
            dev_\rho\left(w_1+\left<\Lambda_{4}^\rho-\Lambda_1^{(0)},\Omega^{-1}\begin{pmatrix}
                1 \\
                0
            \end{pmatrix}\right>,\alpha_2^\rho\xi_2,\alpha_3^\rho\xi_3\right)=\rho(1,0,0)~dev_\rho(w_1,\xi_2,\xi_3)
            ,\\
            dev_\rho\left(w_1+\left<\Lambda_{4}^\rho-\Lambda_1^{(0)},\Omega^{-1}\begin{pmatrix}
                0 \\
                1
            \end{pmatrix}\right>,\beta_2^\rho\xi_2,\beta_3^\rho\xi_3\right)=\rho(0,1,0)~dev_\rho(w_1,\xi_2,\xi_3),\\
            dev_\rho\left(w_1+1,\xi_2,\xi_3\right)=\rho(0,0,1)~dev_\rho(w_1,\xi_2,\xi_3),
        \end{matrix}
        $
    \end{center}
    i.e. $dev_\rho$ is $\rho$-equivariant.
\end{Theorem}

\begin{proof}\label{Preuve Theorem Developping map for the nearby affine structure in the non resonant case}

    We will use the Implicit Function Theorem. 

    For $\rho\in\Hom(\mathbb{Z}^{3},\mathrm{Diag}(\mathbb C^3)^\times)$ close enough to $\rho_0$, we keep the same notations as before, namely $\rho(1,0,0)=A$ and so on, and we consider the map $dev_\rho$ as in the statement of the theorem. By construction, for every $(w_1,\xi_2,\xi_3)\in \mathbb C \times (\mathbb C^2\setminus \{0\})$, we have 

    \begin{center}
        $
        dev_\rho\left(w_1+1,\xi_2,\xi_3\right)=\rho(0,0,1)~dev_{\rho}(w_1,\xi_2,\xi_3).
        $
    \end{center}
    We construct the holomorphic map 

    \begin{center}
        $
        \begin{matrix}
            F & : & \Hom(\mathbb{Z}^{3},\mathrm{Diag}(\mathbb C^3)^\times)\times (\mathbb C^2)^3 & \longrightarrow & \mathbb C^6 \\
              &   & (\rho,\Lambda_4,\Lambda_5,\Lambda_6) & \longmapsto & F(\rho,\Lambda_4,\Lambda_5,\Lambda_6)
        \end{matrix}
        $
    \end{center}
    such that $F(\rho_0,\Lambda_4^{(0)},\Lambda_5^{(0)},\Lambda_6^{(0)})=0$ and such that the six components of $F(\rho,\Lambda_4,\Lambda_5,\Lambda_6)$ measure the defect of $dev_\rho$ to be $\rho$-equivariant. 
    
    The map $F$ satisfies the hypotheses of the Implicit Function Theorem, hence there exists a holomorphic map $\Psi$ from a small neighborhood of $\rho_0$ in $\Hom(\mathbb Z^3,\mathrm{Diag}(\mathbb C ^3)^\times)$ to a small neighborhood of $(\Lambda_4^{(0)},\Lambda_5^{(0)},\Lambda_6^{(0)})$, such that $F(\rho,\Lambda_4,\Lambda_5,\Lambda_6)=0$ holds if and only if $(\Lambda_4,\Lambda_5,\Lambda_6)=\Psi(\rho)$.
    Since the Siegel and weak hyperbolicity conditions (see Definition \ref{Definition Siegel domain and weak hyperbolicity}) are open, we can assume that for every $\rho$, $(\Lambda_1^{(0)},\Lambda_2^{(0)},\Lambda_3^{(0)},\Psi(\rho))$ is still a $(2,6,4)$ LVM configuration.

    The only thing we still have to prove is that the map $\Psi$ is a submersion. This is achieved by writing the equation satisfied by the map $\Psi$.
\end{proof}

\subsection{Deformations of the (G,V)-structure in the resonant case}\label{Subsection Deformations in the resonant case}

Let $(\Lambda_1^{(0)},...,\Lambda_6^{(0)})$ be an LVM configuration of type $(2,6,4)$ that generates an LVM manifold $N_0=V / <A_0,B_0>$ with the same notations as usual. 

We assume, in this section, that $A_0$ and $B_0$ carry the common resonance
$$
\left\{
\begin{matrix}
    \alpha^{(0)}_3 & = & (\alpha^{(0)}_1)^p(\alpha^{(0)}_2)^q \\
    \beta^{(0)}_3 & = & (\beta^{(0)}_1)^p(\beta^{(0)}_2)^q,
\end{matrix}
\right.
$$
where $(p,q)$ is an element of $\Z\times \N^*$.

\begin{Remark}\label{Remark Number of possible resonances}
    According to Lemma \ref{Lemma Classification resonances 264}, the LVM manifold $N_0$ carries another non-trivial resonance if and only if the integer $q$ is equal to $1$. 
\end{Remark}

We recall from Section \ref{Subsection Group of resonant transformations} that the resonances of the LVM manifold generate a group of resonant transformations $G_{p,q}$ of $V=\C^*\times \C^2\setminus\{0\}$. The canonical affine structure of the LVM manifold $N_0$ can be seen as a $(G_{p,q},V)$-structure and, therefore, it is natural to apply the strategy of deforming the $(G_{p,q},V)$-structure and try to understand the underlying deformation space of the complex structure.

Let $f\in G_{p,q}$, respectively $g$, be a resonant transformation in a sufficiently small neighborhood of $A_0$ in $G_{p,q}$, respectively, of $B_0$, such that $f$ and $g$ commute.

\begin{Lemma}\label{Lemma common normal form for commuting resonant transfo}
    Assume that $q=1$. There exists an element $h$ of $G_{p,q}$ such that $h^{-1}fh$ and $h^{-1}gh$ are triangular in the sense of Proposition \ref{Proposition Groupe des transfo résonantes}.
\end{Lemma}

\begin{proof}\label{Preuve Lemma common normal form for commuting resonant transfo}
    We use the notation of the proof of Proposition \ref{Proposition Groupe des transfo résonantes}. Assume that we can write $f=(\alpha,A)$ and $g=(\beta,B)$. The commutation reads 
    $$
    (\alpha\beta,\tau_p(\beta)(A)B)=(\beta\alpha,\tau_p(\alpha)(B)A).
    $$
    We notice that $\tau_p(z)(C)=L_{z,p}CL_{z,p}^{-1}$. Therefore, the commutation is equivalent with
    $$
    L_{\beta,p}AL_{\beta,p}^{-1}B=L_{\alpha,p}BL_{\alpha,p}^{-1}A.
    $$
    Let $\lambda \in \C$ and $z\in\C^2\setminus\{0\}$ be such that $Az=\lambda L_{\alpha,p}z$. Then, we have 
    $$
    AL_{\beta,p}^{-1}Bz=L_{\beta,p}^{-1}L_{\alpha,p}BL_{\alpha,p}^{-1}Az=L_{\beta,p}^{-1}L_{\alpha,p}BL_{\alpha,p}^{-1}\lambda L_{\alpha,p}z=\lambda L_{\alpha,p}L_{\beta,p}^{-1}Bz.
    $$
    Hence, the space $\ker(A-\lambda L_{\alpha,p})$ is stable by $L_{\beta,p}^{-1}B$. Therefore, there exists an eigenvector $y$ of $L_{\beta,p}^{-1}B$ in $\ker(A-\lambda L_{\alpha,p})$. Let $x$ be such that $(x,y)$ is a basis of $\C^2$ and let $P$ be the matrix of this basis in the canonical basis. Then $P$ is such that $\tau_p(\alpha)(P^{-1})AP$ and $\tau_p(\beta)(P^{-1})BP$ are inferior triangular. The element $h=(1,P)$ allows us to conclude the proof.
\end{proof}

From now on, we assume that 
$$
\begin{matrix}
    f & : & V & \longrightarrow & V \\
    && (\xi_1,\xi_2,\xi_3) & \longmapsto & (\alpha_1\xi_1,\alpha_2\xi_2, \alpha_3\xi_3+\varepsilon\xi_1^p\xi_2^q), \\
    g & : & V & \longrightarrow & V \\
    && (\xi_1,\xi_2,\xi_3) & \longmapsto & (\beta_1\xi_1,\beta_2\xi_2, \beta_3\xi_3+\delta\xi_1^p\xi_2^q).
\end{matrix}
$$
The commutation of $f$ and $g$ reads
$$
\varepsilon(\beta_3-\beta_1^p\beta_2^q)=\delta(\alpha_3-\alpha_1^p\alpha_2^q).
$$

\begin{Lemma}\label{Lemma Diagonalisation des elements de Gpq si non resonant}
    Assume that $\alpha_3$ and $\alpha_1^p\alpha_2^q$ are distinct complex numbers. Then there exists an element $h$ in $G_{p,q}$ such that $h^{-1}fh$ and $h^{-1}gh$ are linear diagonal.
\end{Lemma}

\begin{proof}\label{Preuve Lemma Diagonalisation des elements de Gpq si non resonant}
    Let $h$ be defined by 
    $$
    h(\xi_1,\xi_2,\xi_3)=(\xi_1,\xi_2,\xi_3+c\xi_1^p\xi_2^q),
    $$
    where $c$ is a complex number. Then we have
    $$
    \begin{matrix}
        fh(\xi_1,\xi_2,\xi_3) & = & (\alpha_1\xi_1,\alpha_2\xi_2,\alpha_3\xi_3+(\alpha_3c+\varepsilon)\xi_1^p\xi_2^q)
    \end{matrix}
    $$
    and 
    $$
    \begin{matrix}
        h(\alpha_1\xi_1,\alpha_2\xi_2,\alpha_3\xi_3)&=&(\alpha_1\xi_1,\alpha_2\xi_2,\alpha_3\xi_3+c\alpha_1^p\alpha_2^q\xi_1^p\xi_2^q).
    \end{matrix}
    $$
    Therefore, choosing 
    $$
    c=\frac{-\varepsilon}{\alpha_3-\alpha_1^p\alpha_2^q},
    $$
    we have $fh=h(\alpha_1\xi_1,\alpha_2\xi_2,\alpha_3\xi_3)$.

    According to the computations for $f$, we have $gh=h(\beta_1\xi_1,\beta_2\xi_2,\beta_3\xi_3)$ if and only if
    $$
    \begin{matrix}
        \beta_3\frac{-\varepsilon}{\alpha_3-\alpha_1^p\alpha_2^q}+\delta&=&\frac{-\varepsilon}{\alpha_3-\alpha_1^p\alpha_2^q}\beta_1^p\beta_2^q.
    \end{matrix}
    $$
    The latter is equivalent with the commutation of $f$ and $g$.
\end{proof}

We consider the action 

\begin{center}
    $
    \begin{matrix}
        \mathbb{Z}^2\times V & \longrightarrow & V   \\
        ((r,s),\xi) & \longmapsto & f^{\circ r}\circ g^{\circ s}(\xi).
    \end{matrix}
    $
\end{center}

We are now able to prove the following result, which is essential to understand the small deformations of the LVM manifold, the small deformations of the natural resonant structure, and the relation between both.

\begin{Theorem}\label{Theorem The action is fpf and proper}
    The above action is fixed point free, proper and cocompact.
\end{Theorem}

\begin{proof}\label{Preuve Theorem The action is fpf and proper}
    Since $f$, respectively $g$, is close enough to $A_0$, respectively $B_0$, and since the Siegel and weak hyperbolicity conditions are open (see Definition \ref{Definition Siegel domain and weak hyperbolicity}), there exists an LVM configuration of type $(2,6,4)$ such that the LVM manifold induced by this configuration is the quotient of $V$ by the $\mathbb{Z}^2$ action generated by the diagonal matrices 
    \begin{center}
        $
        \begin{matrix}
            A_{diag}=
            \begin{pmatrix}
                \alpha_1 & 0 & 0 \\
                0 & \alpha_2 & 0 \\
                0 & 0 & \alpha_3
            \end{pmatrix}
            &
            \text{and}
            &
            B_{diag}=
            \begin{pmatrix}
                \beta_1 & 0 & 0 \\
                0 & \beta_2 & 0 \\
                0 & 0 & \beta_3
            \end{pmatrix},
        \end{matrix}
        $
    \end{center}
    i.e. these matrices are diagonal and have the same eigenvalues as the linear parts of $f$ and $g$.

    Let $(r,s)\in\mathbb{Z}^2$ and $\xi\in V$ be such that 

    \begin{center}
        $
        f^rg^s(\xi)=\xi.
        $
    \end{center}
    Looking at the form of the map $f^rg^s$, the only way an element $\xi$ of $V$ can satisfy such an equation is that  
    \begin{center}
        $
        \left\{
        \begin{matrix}
            \alpha_1^r\beta_1^s=1,   \\
            \alpha_2^r\beta_2^s=1 \text{ or } \alpha_3^r\beta_3^s=1.
        \end{matrix}
        \right.
        $
    \end{center}
    Therefore, we can easily find an eigenvector for the eigenvalue $1$ in $V$ for the matrix $A_{diag}^rB_{diag}^s$. Since the the action generated by $A_{diag}$ and $B_{diag}$ is fixed point free, we must have $(r,s)=(0,0)$. Hence, the action generated by $A$ and $B$ is fixed point free.

    Now, we are going to prove that the action is proper. If $\varepsilon=\delta=0$, i.e. if $f=A_{diag}$ and $g=B_{diag}$, then we do not have anything to prove, because we know that the action is proper. Therefore, we assume that $\varepsilon$ or $\delta$ is nonzero, and that we are not in the hypotheses of Lemma \ref{Lemma Diagonalisation des elements de Gpq si non resonant}, which implies that $\alpha_3=\alpha_1^p\alpha_2^q$ and $\beta_3=\beta_1^p\beta_2^q$.

    We assume that there exist $(\xi_n)_{n\in\mathbb{N}}\in V^{\mathbb{N}}$ and $(r_n,s_n)_{n\in\mathbb{N}}\in(\mathbb{Z}^2)^{\mathbb{N}}$ such that the sequence $(\xi_n)$ converges toward an element $\xi\in V$ and the sequence $(f^{r_n}g^{s_n}(\xi_n))$ converges also towards an element $\zeta\in V$. We write

    \begin{center}
        $
        \begin{matrix}
            \xi_n=((\xi_n)_1,(\xi_n)_2,(\xi_n)_3),   \\
            \xi=(\xi_1,\xi_2,\xi_3),     \\
            \zeta=(\zeta_1,\zeta_2,\zeta_3).
        \end{matrix}
        $
    \end{center}
    With these notations, we have that $f^{r_n}g^{s_n}(\xi_n)$ is given by
    \vspace{0.1cm}
    \begin{center}
        $
        \begin{matrix}
            \begin{pmatrix}
                \alpha_1^{r_n}\beta_1^{s_n}(\xi_n)_1    \\
                \alpha_2^{r_n}\beta_2^{s_n}(\xi_n)_2   \\
                (\alpha_1^p\alpha_2^q)^{r_n}(\beta_1^p\beta_2^q)^{s_n}(\xi_n)_3 +(\frac{\varepsilon}{\alpha_1^p\alpha_2^q}r_n+\frac{\delta}{\beta_1^p\beta_2^q}s_n)(\alpha_1^p\alpha_2^q)^{r_n}(\beta_1^p\beta_2^q)^{s_n}(\xi_n)_1^p(\xi_n)_2^q
            \end{pmatrix}   &
            \begin{matrix}
                L_1     \\
                L_2     \\
                L_3
            \end{matrix}
        \end{matrix}
        $
    \end{center}   
    
    \vspace{0.1cm}
    Without loss of generality, we can assume, following Lemma \ref{Lemma Useful lemma 2 6 4}, that $\vert\alpha_2\vert<1$, that $\vert\beta_2\vert\leq1$ and that if $\vert\beta_2\vert=1$, then $\vert\beta_1\vert<1$.
    
    \textbf{Step 1: If $(r_n)$ is bounded, then $(s_n)$ is bounded.}
    Since $\xi_1\ne0$ and $\zeta_1\ne0$ because $\xi$ and $\zeta$ are elements of $V=\mathbb{C}^*\times(\mathbb{C}^2\setminus\{0\})$, if $(r_n)$ is bounded, then we must have $\vert\beta_1\vert=1$ or $(s_n)$ bounded. Otherwise, the sequence in $L_1$ cannot converge to $\zeta_1$.

    We assume that $\vert\beta_1\vert=1$, then we must have $\vert\beta_2\vert<1$.
    If we can extract such that, after renaming, $s_n\rightarrow -\infty$, then $\xi_2=0$. If not, the sequence in $L_2$ cannot converge. Again, since $\xi$ is an element of $V$, we have $\xi_3\ne0$, and the term $(\alpha_1^p\alpha_2^q)^{r_n}(\beta_1^p\beta_2^q)^{s_n}(\xi_n)_3$ in $L_3$ is divergent with speed like $\beta_2^{s_n}$. But the second term of $L_3$, which is a $O(s_n)$ (recall that here $(r_n)$ is bounded) cannot make the sequence in $L_3$ converge, which is a contradiction.

    If we can extract such that, after renaming, $s_n\rightarrow +\infty$, then the sequence in $L_2$ converges towards $0$, i.e. $\zeta_2=0$. Since $\zeta$ is an element of $V$, we have $\zeta_3\ne0$. But the sequence in $L_3$ converges towards $0$, which is a contradiction (indeed, the only undetermined term is in $s_n\beta_2^{qs_n}$, but this is a classical comparative growth). 

    \textbf{Step 2: If $(s_n)$ is bounded, then $(r_n)$ is bounded.}
    The argument is almost the same (it is even easier, since we already have $\vert\alpha_2\vert<1$).

    \textbf{Step 3: If neither $(r_n)$ is bounded, nor $(s_n)$ is bounded, then we have a contradiction.}

    In this case, we need to have $\vert\alpha_1\vert\ne1$ and $\vert\beta_1\vert\ne1$, otherwise the sequence in $L_1$ cannot converge towards a nonzero element. Since $\xi_1\ne0$ and $\zeta_1\ne0$, there exists an element $x\in\mathbb{R}$ such that

    \begin{center}
        $
        r_n\ln(\vert\alpha_1\vert)+s_n\ln(\vert\beta_1\vert)=x+o(1).
        $
    \end{center}
    We infer that 

    \begin{center}
        $
        s_n=-\frac{\ln(\vert\alpha_1\vert)}{\ln(\vert\beta_1\vert)}r_n+\frac{x}{\ln(\vert\beta_1\vert)}+o(1).
        $
    \end{center}
    Using that, we compute 
    $$
    \begin{matrix}
        \vert\alpha_2^{r_n}\beta_2^{s_n}\vert   &
        =   &
        \vert\alpha_2\vert^{r_n}\vert\beta_2\vert^{-\frac{\ln(\vert\alpha_1\vert)}{\ln(\vert\beta_1\vert)}r_n+\frac{x}{\ln(\vert\beta_1\vert)}+o(1)}    \\
            &
        =   &
        \vert\alpha_2\beta_2^{-\frac{\ln(\vert\alpha_1\vert)}{\ln(\vert\beta_1\vert)}}\vert^{r_n}\vert\beta_2\vert^{\frac{x}{\ln(\vert\beta_1\vert)}}(1+o(1))        \\
            &
        \sim    &
        K_1\vert\alpha_2\beta_2^{-\frac{\ln(\vert\alpha_1\vert)}{\ln(\vert\beta_1\vert)}}\vert^{r_n},
    \end{matrix}
    $$
    and, the same way,
    $$
    \begin{matrix}
        \vert(\alpha_1^p\alpha_2^q)^{r_n}(\beta_1^p\beta_2^q)^{s_n}\vert   &
        \sim    &
        K_2\vert\alpha_2\beta_2^{-\frac{\ln(\vert\alpha_1\vert)}{\ln(\vert\beta_1\vert)}}\vert^{qr_n},
    \end{matrix}
    $$
    with $K_1,K_2$ positive real numbers, because $\vert\alpha_1^{pr_n}\beta_1^{qs_n}\vert$ has to converge to a positive real number, in order that $L_1$ converges to $\zeta_1\neq0$.

    We assume that we can extract such that, after renaming, $r_n\rightarrow +\infty$. 

    \textbf{Case 1:}

    If $\vert\alpha_2\beta_2^{-\frac{\ln(\vert\alpha_1\vert)}{\ln(\vert\beta_1\vert)}}\vert<1$, then the sequences in $L_2$ and $L_3$ converge both to $0$, which is impossible since $\zeta$ is an element of $V$. The convergence towards zero is easy to see in $L_2$, and for $L_3$, we use the fact that $s_n$ is equivalent to $-\frac{\ln(\vert\alpha_1\vert)}{\ln(\vert\beta_1\vert)}r_n$ and comparative growth for terms like $r_n\vert\alpha_2\beta_2^{-\frac{\ln(\vert\alpha_1\vert)}{\ln(\vert\beta_1\vert)}}\vert^{qr_n}$.

    \textbf{Case 2:}

    If $\vert\alpha_2\beta_2^{-\frac{\ln(\vert\alpha_1\vert)}{\ln(\vert\beta_1\vert)}}\vert>1$, then $\xi_2=0$, otherwise the sequence in $L_2$ cannot converge. Again, since $\xi$ is an element of $V$, we have $\xi_3\ne0$, and the term $(\alpha_1^p\alpha_2^q)^{r_n}(\beta_1^p\beta_2^q)^{s_n}(\xi_n)_3$ in $L_3$ is divergent with speed like $\vert\alpha_2\beta_2^{-\frac{\ln(\vert\alpha_1\vert)}{\ln(\vert\beta_1\vert)}}\vert^{qr_n}$. But the second term of $L_3$, which is a $O(r_n)$ (again using the fact that $s_n$ is equivalent to $-\frac{\ln(\vert\alpha_1\vert)}{\ln(\vert\beta_1\vert)}r_n$) cannot make the sequence in $L_3$ converge, which is a contradiction.

    \textbf{Case 3:}

    If $\vert\alpha_2\beta_2^{-\frac{\ln(\vert\alpha_1\vert)}{\ln(\vert\beta_1\vert)}}\vert=1$, then $\alpha_2^{r_n}\beta_2^{s_n}$ and $(\alpha_1^p\alpha_2^q)^{r_n}(\beta_1^p\beta_2^q)^{s_n}=\alpha_3^{r_n}\beta_3^{s_n}$ have converging moduli, so we consider the sequence $(\xi'_n)_{n\in\mathbb{N}}\in V^{\mathbb{N}}$ such that $\xi'_n=((\xi_n)_1,1,1)$ and we have 

    \begin{center}
        $
        \begin{matrix}
            (\xi'_n) $ converges towards $(\xi_1,1,1)\in V,      \\
            r_n\rightarrow +\infty,      \\
            (A_{diag}^{r_n}B_{diag}^{s_n}\xi'_n) $ is included in a compact subset of $ V .
        \end{matrix}
        $
    \end{center}
    This contradicts the fact that the action of the group generated by $A_{diag}$ and $B_{diag}$ on $V$ is proper.
    
    If we can extract such that, after renaming, $r_n\rightarrow -\infty$, then the same argument holds. 
    
    To conclude, we have shown in \textbf{Step 3} that $(r_n)$ or $(s_n)$ is bounded, and according to \textbf{Step 1} and \textbf{Step 2}, they both are bounded, showing that the action of the group generated by $A$ and $B$ on $V$ is proper.

    Finally, we need to prove that the quotient space is compact. We will use the Ehresmann-Thurston principle, providing a nice and simple proof of compactness. Up to a linear diagonal change of coordinate in $\mathbb C^3$, we can assume that $\varepsilon$ and $\delta$ are small. Then the homomorphism

    \begin{center}
    $
    \begin{matrix}
        \rho:    &    \mathbb{Z}^{3}    &   \longrightarrow     &   G_{p,q}   \\
            &   (t,r,s)   &   \longmapsto     &   f^rg^s.
    \end{matrix}
    $
    \end{center}
    is close to the initial homomorphism that is associated with the standard $(G_{p,q},V)$-structure of the LVM manifold $N_d:=V/<A_{diag},B_{diag}>$. By the Ehresmann-Thurston principle, the flat bundle over $N_d$
    $$
    \begin{matrix}
        ((\mathbb C\times (\mathbb{C}^2\setminus \{0\}))\times V)/\mathbb Z^3 & \longrightarrow &  N_d
    \end{matrix}
    $$
    given by $\rho$ admits a smooth global section $s$ that is transverse to the flat structure. On the other hand, the total space of this bundle has a natural projection $\varpi$ to $V/<f,g>$. According to what we proved above, this last space has a natural structure of complex manifold, and $\varpi$ is a smooth (even holomorphic) bundle. The composition map $\varpi\circ s : N_d\to V/<f,g>$ is smooth. But since $s$ is transverse to the flat structure of $\varpi$, $\varpi\circ s$ is a local diffeomorphism. Since $V$ is connected, so is $V/<f,g>$ and, therefore the map $\varpi\circ s$ is onto. Hence, the compactness of $N_d$ implies the compactness of $V/<f,g>$.
\end{proof}

This result presents new complex structures, which are arbitrarily close to the LVM one, and gives a description of some resonant structures arbitrary close to the natural one. To be more specific, we infer that every small deformation of the natural resonant structure such that the fundamental group of $V$ has trivial holonomy is complete, in the sense of Definition \ref{Definition completeness and Kleinian for GX structures}. 

\subsection{Kuranishi family in the resonant case}\label{Subsection Kuranishi family of a resonant LVM manifold}

In this section, we will describe the Kuranishi family of an LVM manifold of type $(2,6,4)$. In Section \ref{Subsection Kuranishi family of a non resonant (2,6,4) LVM manifold}, we deal with the case of non-resonant manifolds. We show that the Kuranishi space in that case is smooth and that the Kuranishi family is the LVM family constructed in \cite{Mee2000}. We keep the notations of the previous section \ref{Subsection Deformations in the resonant case}.

In the case where $q\geq 2$, we will represent an element of the group $G_{p,q}$ by a $3\times3$ matrix. We consider the algebraic variety
\footnotesize
$$
\widetilde T_{p,q}:=
\left\{
\left(
\begin{pmatrix}
    \alpha_1 & 0 & 0 \\
    0 & \alpha_2 & 0 \\
    0 & \varepsilon & \alpha_3
\end{pmatrix},
\begin{pmatrix}
    \beta_1 & 0 & 0 \\
    0 & \beta_2 & 0 \\
    0 & \delta & \beta_3
\end{pmatrix}
\right)\in G_{p,q}^2,~\varepsilon(\beta_3-\beta_1^p\beta_2^q)=\delta(\alpha_3-\alpha_1^p\alpha_2^q)
\right\}.
$$
\normalsize
According to Theorem \ref{Theorem The action is fpf and proper}, there exists a neighborhood $T_{p,q}$ of $(A_0,B_0)$ in $\widetilde T_{p,q}$ such that for every element $(A,B)$ in $T_{p,q}$ with 
$$
A=
\begin{pmatrix}
    \alpha_1 & 0 & 0 \\
    0 & \alpha_2 & 0 \\
    0 & \varepsilon & \alpha_3
\end{pmatrix}~B=
\begin{pmatrix}
    \beta_1 & 0 & 0 \\
    0 & \beta_2 & 0 \\
    0 & \delta & \beta_3
\end{pmatrix},
$$
the action $\mathcal A^{p,q}_{A,B}:\Z^2\times V \to V$ of $\Z^2$ on $V$ given by 
$$
\begin{matrix}
    (1,0).(\xi_1,\xi_2,\xi_3)=(\alpha_1\xi_1,\alpha_2\xi_2,\alpha_3\xi_3+\varepsilon\xi_1^p\xi_2^q), \\
    (0,1).(\xi_1,\xi_2,\xi_3)=(\beta_1\xi_1,\beta_2\xi_2,\beta_3\xi_3+\delta\xi_1^p\xi_2^q)
\end{matrix}
$$
is fixed point free, proper and cocompact. 

We consider the action of $\Z^2$
$$
\begin{matrix}
    \Z^2\times T_{p,q}\times V & \longrightarrow & T_{p,q}\times V \\
    ((r,s),(A,B),\xi) & \longmapsto & ((A,B),\mathcal A^{p,q}_{A,B}((r,s),\xi)).
\end{matrix}
$$
Obviously, this action is fixed point free and proper. Therefore, the quotient space $\mathcal M_{p,q}:=(T_{p,q}\times V)/\mathbb Z^2$ carries a natural structure of algebraic variety, and the natural projection $T_{p,q}\times V \to T_{p,q}$ descends to the quotient to a flat and proper algebraic map $\varpi_{p,q} :\mathcal M_{p,q} \to T_{p,q}$. The fibers of $\varpi_{p,q}$ are compact complex manifolds and $\varpi_{p,q}^{-1}(A_0,B_0)$ is isomorphic to $N_0$, so $\varpi_{p,q}$ is a family of deformations of $N_0$.

In the case where $q=1$, we consider the algebraic variety
\scriptsize
$$
\widetilde S_{p}:=
\left\{
\left(
\begin{pmatrix}
    \alpha_1 & 0 & 0 \\
    0 & \alpha_2 & \varepsilon_2 \\
    0 & \varepsilon_1 & \alpha_3
\end{pmatrix},
\begin{pmatrix}
    \beta_1 & 0 & 0 \\
    0 & \beta_2 & \delta_2 \\
    0 & \delta_1 & \beta_3
\end{pmatrix}
\right)\in G_{p,q}^2,~
\left\vert
\begin{matrix}
\varepsilon_1\delta_2\beta_1^{p}&=&\delta_1\varepsilon_2\alpha_1^{p} \\
    \varepsilon_1(\beta_3-\beta_1^p\beta_2)&=&\delta_1(\alpha_3-\alpha_1^p\alpha_2) \\
    \varepsilon_2(\beta_2-\beta_1^{-p}\beta_3)&=&\delta_2(\alpha_2-\alpha_1^{-p}\alpha_3)
\end{matrix}
\right.
\right\}.
$$
\normalsize
According to Proposition \ref{Proposition Groupe des transfo résonantes} \textit{(iii)} and Theorem \ref{Theorem The action is fpf and proper}, there exists a neighborhood $S_p$ of $(A_0,B_0)$ in $\widetilde S_{p}$ such that for every element $(A,B)$ in $S_p$ with 
$$
A=
\begin{pmatrix}
    \alpha_1 & 0 & 0 \\
    0 & \alpha_2 & \varepsilon_2 \\
    0 & \varepsilon_1 & \alpha_3
\end{pmatrix}~B=
\begin{pmatrix}
    \beta_1 & 0 & 0 \\
    0 & \beta_2 & \delta_2 \\
    0 & \delta_1 & \beta_3
\end{pmatrix},
$$
the action $\mathcal A^{p}_{A,B}:\Z^2\times V \to V$ of $\Z^2$ on $V$ given by 
$$
\begin{matrix}
    (1,0).(\xi_1,\xi_2,\xi_3)&=&(\alpha_1\xi_1,\alpha_2\xi_2+\varepsilon_2\xi_1^{-p}\xi_3,\alpha_3\xi_3+\varepsilon_1\xi_1^p\xi_2), \\
    (0,1).(\xi_1,\xi_2,\xi_3)&=&(\beta_1\xi_1,\beta_2\xi_2+\delta_2\xi_1^{-p}\xi_3,\beta_3\xi_3+\delta_1\xi_1^p\xi_2)
\end{matrix}
$$
is fixed point free, proper and cocompact. 

We consider the action of $\Z^2$
$$
\begin{matrix}
    \Z^2\times S_p\times V & \longrightarrow & S_p\times V \\
    ((r,s),(A,B),\xi) & \longmapsto & ((A,B),\mathcal A^{p}_{A,B}((r,s),\xi)).
\end{matrix}
$$
Obviously, this action is fixed point free and proper. Therefore, the quotient space $\mathcal M_{p}:=(S_{p}\times V)/\mathbb Z^2$ carries a natural structure of algebraic variety, and the natural projection $S_{p}\times V \to S_{p}$ descends to the quotient to a flat and proper algebraic map $\varpi_{p} :\mathcal M_{p} \to S_{p}$. The fibers of $\varpi_{p}$ are compact complex manifolds and $\varpi_{p}^{-1}(A_0,B_0)$ is isomorphic to $N_0$, so $\varpi_{p}$ is a family of deformations of $N_0$.

The purpose of this section is to prove that $\varpi_{p,q}$, or $\varpi_p$ if $q=1$, is the Kuranishi family of $N_0$. We begin with some lemmas concerning the Kodaira-Spencer map of this family.

\begin{Lemma}\label{Lemma Zariski tangent space of Tpq and Sp}
    (i) The Zariski tangent space of $T_{p,q}$, or $S_p$, depending on the case, at $(A_0,B_0)$ is isomorphic to the space 
    $$
    \left(VF^{\alpha^{(0)},\beta^{(0)}}_{\mathbb Z \times \mathbb N \times \mathbb N}\right)^2.
    $$
    (ii) The Kodaira-Spencer map of the family $\varpi_{p,q}$, or $\varpi_p$, is an isomorphism.
\end{Lemma}

\begin{proof}\label{Preuve Lemma Zariski tangent space of Tpq and Sp}
    \textit{(i)} The affine varieties $\widetilde{T}_{p,q}$ and $\widetilde{S}_{p,q}$ can both be written as the representation variety 
    $$
    \Hom(\mathbb Z^2,G_{p,q}).
    $$
    We denote by $\mathfrak g_{p,q}$ the Lie algebra of the group $G_{p,q}$. The interest of this writing is that we have an isomorphism 
    $$
    T_\rho \widetilde{T}_{p,q} \cong H^1(\mathbb Z^2,(\mathfrak g_{p,q})_{Ad_\rho})
    $$ 
    or 
    $$
    S_\rho \widetilde{S}_{p} \cong H^1(\mathbb Z^2,(\mathfrak g_{p,q})_{Ad_\rho}),
    $$
    where $H^1(\mathbb Z^2,(\mathfrak g_{p,q}) _{Ad_\rho})$ is the first group cohomology induced by $\rho\in \widetilde{T}_{p,q}$, or $\widetilde{S}_{p}$.

    Due to the resonance assumption, the natural homomorphism $\rho_0=(A_0,B_0)$ takes values in the center of $G_{p,q}$. Therefore, we have 
    $$
    H^1(\mathbb Z^2,(\mathfrak g_{p,q}) _{Ad_{\rho_0}})\cong \mathfrak g_{p,q} ^2,
    $$ which concludes the proof.

    \textit{(ii)} We will use our computation of the first obstruction, namely Proposition \ref{Proposition First obstruction}. Indeed, we will show that for every $(X,Y)\in \left(VF^{\alpha^{(0)},\beta^{(0)}}_{\mathbb Z \times \mathbb N \times \mathbb N}\right)^2$ such that $[X,Y]=0$, there exists a germ of holomorphic curve $\eta :(\mathbb C,0)\to\mathbb (T_{p,q},\rho_0)$ or $\eta:(\mathbb C,0)\to\mathbb (S_{p},\rho_0)$ such that the Kodaira-Spencer map of the pullback family over $(\mathbb C,0)$ denoted by $KS_{\eta^*\varpi_{p,q}}$ verifies 
    $$
    KS_{\eta^*\varpi_{p,q}}\left(\frac{\partial}{\partial t}\right)=(X,Y).
    $$

    Let $(X,Y)\in \left(VF^{\alpha^{(0)},\beta^{(0)}}_{\mathbb Z \times \mathbb N \times \mathbb N}\right)^2$ be such that $[X,Y]=0$. We consider 
    $$
    \begin{matrix}
        \eta & : & (\mathbb C,0) & \longrightarrow & \mathbb (T_{p,q},\rho_0) ~(respectively ~ (S_p,\rho_0)) \\
             &   &  t            & \longmapsto     & (A_0\exp(tX),B_0\exp(tY)),
    \end{matrix}
    $$ 
    where we identify a homomorphism from $\mathbb Z^2$ to $G_{p,q}$ with the couple of images of the two standard generators of $\mathbb Z^2$ and the elements of $G_{p,q}$ with their standard action by biholomorphism on $V=\mathbb C^* \times (\mathbb C^2\setminus \{0\})$.
    Performing the same computation as in Section \ref{Subsection Kuranishi family of a non resonant (2,6,4) LVM manifold}, we get exactly what we wanted.

    This shows that the Kodaira-Spencer map of the family $\varpi_{p,q}$ (or $\varpi_p$) is onto. According to \textit{(i)}, the spaces $T_{\rho_0}T_{p,q}$, $T_{\rho_0}S_p$ and $H^1(N_0,\Theta)$ have the same dimension. Therefore, the Kodaira-Spencer map of the family $\varpi_{p,q}$ or $\varpi_p$ is an isomorphism.
\end{proof}

If the space $T_{p,q}$ (or $S_p$) was smooth, according to Corollary \ref{Corollary Smooth family with invertible KS map}, we could conclude that $\varpi_{p,q}$ (or $\varpi_p$) is the Kuranishi family of the manifold $N_0$. But this is not the case. Still, $\varpi_{p,q}$ (or $\varpi_p$) can be the Kuranishi family of $N_0$, because we showed that the Kuranishi space of $N_0$ is not smooth, because there are obstructed elements in $H^1(N_0,\Theta)$. On the other hand, we have the following corollary. 

\begin{Corollary}\label{Corollary The higher obstructions vanish}
    Let $\theta\in H^1(N_0,\Theta)$ be such that $[\theta,\theta]=0$, i.e. the first obstruction on $\theta$ vanishes. Then all the obstructions on $\theta$ vanish.
\end{Corollary}

\begin{proof}\label{Preuve Corollary The higher obstructions vanish}
    According to the proof of Lemma \ref{Lemma Zariski tangent space of Tpq and Sp}, we can find a one-parameter family of deformations of $N_0$ that realizes the infinitesimal deformation $\theta$.
\end{proof}

In order to show that $\varpi_{p,q}$ (or $\varpi_p$) is the Kuranishi family of $N_0$, it remains to be shown that this family is complete. We will adapt an argument due to Ghys in his paper \cite{Ghys95}. 

We define a sheaf of non-abelian groups on $N_0$ as follows. For $U$ an open subset of $N_0$, we consider all the biholomorphisms $w:W\to W'$ between two open neighborhoods of $U\times \{0\}$ in $N_0\times \mathbb C$ such that, for every $(x,t)$ in $W$, we have $w(x,t)=(w_1(x,t),t)$ and $w_1(x,0)=x$. We identify two such biholomorphisms if they agree on restrictions. The set of equivalence classes for this relation is denoted by $\Lambda(U)$ and those non-abelian groups (for the composition) are the groups of sections of a sheaf of non-abelian groups $\Lambda$.

\begin{Proposition}\label{Proposition Interet du faisceau Lambda}
    (\cite{Dou60Obs}, Proposition 1 p.4-10) The set $H^1(N_0,\Lambda)$ is identified with the set of germs of deformations of $N_0$ over $(\mathbb C,0)$.
\end{Proposition}

For $U$ an open subset of $N_0$, the group $\Lambda(U)$ is endowed with a natural filtration. Indeed, for $k\geq 1$, let $\Lambda_k(U)$ be the group of (classes of) biholomorphisms in $\Lambda(U)$ that are tangent to identity up to order $k-1$ along $U\times \{0\}$. Then, we have
$$
...\subset \Lambda_2(U) \subset \Lambda_1(U)=\Lambda(U).
$$

\begin{Proposition}\label{Proposition Iso quotients Lambda avec Theta}
    \cite{Dou60Obs}, p.4-12) For every $k\geq 1$, the quotient sheaf $A_k:=\Lambda_k/\Lambda_{k+1}$ is isomorphic to $\Theta$, the sheaf of germs of holomorphic vector fields on $N_0$.
\end{Proposition}

\begin{proof}\label{Preuve Proposition Iso quotients Lambda avec Theta}
    Let $U$ be an open subset of $N_0$. For every holomorphic vector field $\chi$ on $U$, we denote by $\phi_\chi:W_\chi  \to U$ the flow of $\chi$, where $(U\times \{0\})\subset W_\chi \subset (U\times \mathbb C)$. 

    For every $k\geq 1$, considering flows allows us to define the map
    $$
    \begin{matrix}
        \Theta(U) & \longrightarrow & \Lambda_k(U)/\Lambda_{k+1}(U) \\
        \chi      & \longmapsto     & [w_\chi:W_\chi\to w_\chi(W_\chi);(x,t)\mapsto (\phi_\chi(x,t^k),t)].
    \end{matrix}
    $$
    For every $U$ this map is an isomorphism. Therefore, this collection of maps produces an isomorphism of sheaves.
\end{proof}

For every $k\geq0$, we denote by $Q_k$ the quotient sheaf $\Lambda/\Lambda_{k+1}$ and we have the following short exact sequence of sheaves of groups
$$
\begin{tikzcd}
    0 \arrow[r] & A_{k+1} \arrow[r] & Q_{k+1} \arrow[r] & Q_k \arrow[r] & 0.
\end{tikzcd}
$$

The important remark of Ghys is that, in the definition of $\Lambda$ and all the sheaves we derived, we could have restricted ourselves to local biholomorphisms in a prescribed sub-pseudogroup of the pseudogroup of local biholomorphisms of $N_0$. In his article \cite{Ghys95}, Ghys considers, for example, the pseudogroup of local biholomorphisms of $\mathrm{SL}(2,\mathbb C)/\Gamma$ that lift to $\mathrm{SL}(2,\mathbb C)$ as restrictions of left translations (p.125).

Here we will use the same technique. We denote by $N^{\mathrm{diff}}$ the smooth manifold underlying to $N_0$.

Let $U$ be an open subset of $N_0$. We consider all the biholomorphisms $w:W\to W'$ between two open neighborhoods of $U\times \{0\}$ in $N_0\times \mathbb C$ such that, for every $(x,t)$ in $W$, we have that 
$w(x,t)=(w_1(x,t),t)$ and $w_1(x,0)=x$ and such that $w$ lifts to an open subset of $V\times \mathbb C$ as a map of the form
$$
\begin{matrix}
    (\xi,t) & \longmapsto & (g(t)\xi,t),
\end{matrix}
$$ 
where $g(t)$ is an element of $G_{p,q}$. As before, we identify two such biholomorphisms if they agree on restrictions. The set of equivalence classes for this relation is denoted by $\Lambda^{G_{p,q}}(U)$ and these non-abelian groups (for the composition) are the groups of sections of a sheaf of non-abelian groups $\Lambda^{G_{p,q}}$. The purpose of this sheaf is that, according to the Ehresmann-Thurston principle \ref{Theorem Ehresmann Thurston principle} applied to $(G_{p,q},V)$-structures, the set $H^1(N_0,\Lambda^{G_{p,q}})$ is identified with the set of germs at $0$ of holomorphic curves in $(\Hom(\mathbb Z^3,G_{p,q}),\rho_0)$. For every $k\geq1$, we consider $\Lambda_k^{G_{p,q}}$, $A_k^{G_{p,q}}$ and $Q_{k-1}^{G_{p,q}}$ defined with $\Lambda^{G_{p,q}}$ the same way as before.

We will need the following lemma.

\begin{Lemma}\label{Lemma map between the representations varieties for Gpq}
    There exist an open neighborhood $U$ of $(A_0,B_0,Id)$ in the representation variety $\Hom(\Z^3,G_{p,q})$, an open neighborhood $W$ of $(A_0,B_0)$ in the variety $\Hom(\Z^2,G_{p,q})$ and an analytic fibration with smooth fibers 
    $$
    \Psi:U\longrightarrow W
    $$
    such that for every $\rho$ in $U$, there exists a $(G_{p,q},V)$-structure on $N^{\mathrm{diff}}$ with holonomy homomorphism $\rho$ close to the initial one, and such that, if $(A,B)=\Psi(\rho)$, then the complex structure induced by the $(G_{p,q},V)$-structure is given by the quotient of $V$ by the action $\mathcal A^{p,q}_{A,B}$.
\end{Lemma}

\begin{proof}\label{Preuve Lemma map between the representations varieties for Gpq}
    According to the Ehresmann-Thurston principle \ref{Theorem Ehresmann Thurston principle}, there exists an open neighborhood $U_1$ of $(A_0,B_0,Id)$ in $\Hom(\mathbb Z^3,G_{p,q})$ such that for every $\rho$ in $U_1$, there exists a $(G_{p,q},V)$-structure on $N^{\mathrm{diff}}$ close to the initial one, and with holonomy morphism $\rho$. 

    Shrinking $U_1$ if needed, every $\rho$ in $U_1$ is written
    $$
    \left(
    \begin{pmatrix}
        \alpha & (0) \\
        (0)    & A
    \end{pmatrix},
    \begin{pmatrix}
        \beta & (0) \\
        (0)    & B
    \end{pmatrix},
    \begin{pmatrix}
        \gamma & (0) \\
        (0)    & C
    \end{pmatrix}
   \right),
    $$
    where 
    $$
    \begin{pmatrix}
        \alpha & (0) \\
        (0)    & A
    \end{pmatrix}
    $$
    is close enough to $A_0$ in the group $G_{p,q}$,
    $$
    \begin{pmatrix}
        \beta & (0) \\
        (0)    & B
    \end{pmatrix}
    $$
    is close enough to $B_0$, $\gamma$ is close enough to $1$ so that its classical logarithm is defined and 
    $$
    C=\begin{pmatrix}
        c_1 & c_2 \\
        c_3 & c_4
    \end{pmatrix}
    $$ 
    is close enough to the identity matrix so that its classical logarithm is also well defined and, of course, verifies $(q-1)c_2=0$.

    In case we have $q=1$, we consider the map 
    $$
    \begin{matrix}
        \mathrm{Dev}_{\gamma,C} & : & \mathbb C \times (\mathbb C^2\setminus \{0\}) & \longrightarrow & V \\
                       &   & (w_1,\xi_2,\xi_3) & \longmapsto & \left(\exp_{G_{p,q}}\left(w_1\log_{G_{p,q}}\begin{pmatrix}
                           \gamma & (0) \\
                           (0) & C
                       \end{pmatrix}\right).\left(e^{2i\pi w_1},\xi_2,\xi_3\right)\right),
    \end{matrix}
    $$
    where $\exp_{G_{p,q}}$ is the exponential map of the Lie group $G_{p,q}$ and $\log_{G_{p,q}}$ its local inverse around $0\in\mathfrak g_{p,q}$ (recall that $\gamma$ is close enough to $1$ and that $C$ is close enough to $Id$).

    This map is made such that it satisfies the following equivariance relation:  
    $$
    \mathrm{Dev}_{\gamma,C}(w_1+1,\xi_2,\xi_3)=\begin{pmatrix}
        \gamma & (0) \\
        (0)    & C
    \end{pmatrix}.\mathrm{Dev}_{\gamma,C}(w_1,\xi_2,\xi_3).
    $$
    We would like this map to be the developing map of the $(G_{p,q},V)$-structure given by $\rho$, so, writing the equivariance conditions, we define $\delta$ and $\eta$ by the formulae
    $$
    \begin{matrix}
        \delta \exp\left(\frac{1}{2i\pi}\log(\gamma)\log(\delta)\right)&=&\alpha, \\
        \eta \exp\left(\frac{1}{2i\pi}\log(\gamma)\log(\eta)\right)&=&\beta, 
    \end{matrix}
    $$
    and we define also 
    $$
    \begin{matrix}
        D & = & \tau_p(\delta)(N_{\delta,\gamma,C}^{-1})A, \\
        E & = & \tau_p(\eta)(N_{\eta,\gamma,C}^{-1})B.
    \end{matrix}
    $$
    where $N_{x,\gamma,C}$ is the $2\times2$-matrix such that 
    $$
    \exp_{G_{p,q}}\left(\frac{1}{2i\pi}\log(x)\log_{G_{p,q}}\begin{pmatrix}
        \gamma & (0) \\
        (0) & C
    \end{pmatrix}\right)=\begin{pmatrix}
        e^{\frac{1}{2i\pi}\log(x)\log(\gamma)} & (0) \\
        (0) & N_{x,\gamma,C}
    \end{pmatrix}.
    $$
    The fact that $\delta$ and $\eta$ are well defined is a simple application of the Implicit Functions Theorem. 
    
    In case we have $q\geq2$ and $c_4\neq \gamma^pc_1^q$, we consider the map 
    \footnotesize
    $$
    \begin{matrix}
        \mathrm{Dev}_{\gamma,C} & : & \mathbb C \times (\mathbb C^2\setminus \{0\}) & \longrightarrow & V \\
                       &   & (w_1,\xi_2,\xi_3) & \longmapsto & 
                       \begin{pmatrix}
                           e^{(2i\pi+\log(\gamma))w_1} \\
                           e^{w_1\log(c_1)}\xi_2 \\
                           e^{w_1\log(c_4)}\xi_3+\frac{c_3}{\gamma^pc_1^q-c_4}e^{p(2i\pi+\log(\gamma))w_1}e^{qw_1\log(c_1)}\xi_2^q
                       \end{pmatrix},
    \end{matrix}
    $$
    \normalsize
    and we define $\delta, D, \eta$ and $E$ as follows:
    $$
    \begin{matrix}
        \delta \exp\left(\frac{1}{2i\pi}\log(\gamma)\log(\delta)\right)&=&\alpha, \\
        d_1 & = & a_1\exp\left(-\frac{1}{2i\pi}\log(\delta)\log(c_1)\right),\\
        d_4 & = & a_4\exp\left(-\frac{1}{2i\pi}\log(\delta)\log(c_4)\right),\\
        D & = & \begin{pmatrix}
            d_1 & 0 \\
            0 & d_4
        \end{pmatrix}, \\
        \eta \exp\left(\frac{1}{2i\pi}\log(\gamma)\log(\eta)\right)&=&\beta, \\
        e_1 & = & b_1\exp\left(-\frac{1}{2i\pi}\log(\delta)\log(c_1)\right),\\
        e_4 & = & b_4\exp\left(-\frac{1}{2i\pi}\log(\delta)\log(c_4)\right),\\
        
        E & = & \begin{pmatrix}
            e_1 & 0 \\
            0 & e_4
        \end{pmatrix}.
    \end{matrix}
    $$
    
    In case we have $q\geq2$ and $c_4= \gamma^pc_1^q$, we consider the map 
    \footnotesize
    $$
    \begin{matrix}
        \mathrm{Dev}_{\gamma,C} & : & \mathbb C \times (\mathbb C^2\setminus \{0\}) & \longrightarrow & V \\
                       &   & (w_1,\xi_2,\xi_3) & \longmapsto & 
                       \begin{pmatrix}
                           e^{(2i\pi+\log(\gamma))w_1}\\
                           e^{w_1\log(c_1)}\xi_2 \\
                           e^{w_1\log(c_4)}\left(\xi_3+\frac{c_3}{c_4}w_1e^{2ip\pi w_1}\xi_2^q\right)
                       \end{pmatrix},
    \end{matrix}
    $$
    \normalsize
    and we define $\delta, D, \eta$ and $E$ as follows:
    $$
    \begin{matrix}
        \delta \exp\left(\frac{1}{2i\pi}\log(\gamma)\log(\delta)\right)&=&\alpha, \\
        d_1 & = & a_1\exp\left(-\frac{1}{2i\pi}\log(\delta)\log(c_1)\right),\\
        d_4 & = & a_4\exp\left(-\frac{1}{2i\pi}\log(\delta)\log(c_4)\right),\\
        d_3 & = & a_3\frac{d_4}{a_4}-\frac{c_3}{c_4}.\frac{1}{2i\pi}\log(\delta)\delta^pd_1^q, \\
        D & = & \begin{pmatrix}
            d_1 & 0 \\
            d_3 & d_4
        \end{pmatrix}, \\
        \eta \exp\left(\frac{1}{2i\pi}\log(\gamma)\log(\eta)\right)&=&\beta, \\
        e_1 & = & b_1\exp\left(-\frac{1}{2i\pi}\log(\eta)\log(c_1)\right),\\
        e_4 & = & b_4\exp\left(-\frac{1}{2i\pi}\log(\eta)\log(c_4)\right),\\
        e_3 & = & b_3\frac{e_4}{b_4}-\frac{c_3}{c_4}.\frac{1}{2i\pi}\log(\eta)\eta^pe_1^q, \\
        E & = & \begin{pmatrix}
            e_1 & 0 \\
            e_3 & e_4
        \end{pmatrix}.
    \end{matrix}
    $$
    
    In all the cases mentioned above, the element 
    $$
    \Psi(\rho):=\left(
    \begin{pmatrix}
        \delta& (0) \\
        (0) & D
    \end{pmatrix},
    \begin{pmatrix}
        \eta & (0)\\
        (0) & E
    \end{pmatrix}
    \right)
    $$
    of $\Hom(\Z^2,G_{p,q})$ generates a $\Z^2$ action on $V$ that is fixed point free, proper and cocompact, according to Theorem \ref{Theorem The action is fpf and proper}, and $\mathrm{Dev}_{\gamma,C}$ is the holomorphic developing map of the Ehresmann-Thurston structure on $N^{\mathrm{diff}}$ with holonomy $\rho$ and underlying complex structure given by the quotient $V/\Z^2$ generated by $\Psi(\rho)$.
    
    We notice that $\Psi$ can be defined on an open subset of the (smooth) complex manifold $G_{p,q}^3$ and is a holomorphic submersion onto an open subset of $G_{p,q}^2$ that maps an element in $\Hom(\Z^3,G_{p,q})$ on an element of $\Hom(\Z^2,G_{p,q})$. Moreover, locally around $(A_0,B_0)$ in $\Hom(\Z^2,G_{p,q})$, every element $\rho\in\Hom(\Z^2,G_{p,q})$ is the image of $(\rho,Id)\in\Hom(\Z^3,G_{p,q})$. Therefore, $\Psi$ is an analytic fibration with smooth fibers.
    \end{proof}

In that context, Proposition \ref{Proposition Cohomology in Theta for a LVM manifold of type 264} can be reformulated in our context as follows.

\begin{Proposition}\label{Proposition isomorphism cohomologie AGpq et A}
    For every $k\geq 1$ and every $j\in\{1,2\}$, the map 
    $$
    H^j(N_0,A_k^{G_{p,q}})\longrightarrow H^j(N_0,A_k)
    $$
    is an isomorphism.
\end{Proposition}

\begin{Lemma}\label{Lemma iso section globales LAMBDA Gpq et LAMBDA}
    The map 
    $$
    H^0(N_0,\Lambda^{G_{p,q}})\longrightarrow H^0(N_0,\Lambda)
    $$
    is an isomorphism.
\end{Lemma}

For $k=1$, the map 
$$
H^1(N_0,Q_{1}^{G_{p,q}}) \longrightarrow H^1(N_0,Q_{1})
$$
is one-to-one onto its image, since $Q_1=A_1$ and $Q_1^{G_{p,q}}=A_1^{G_{p,q}}$, so we can apply Proposition \ref{Proposition isomorphism cohomologie AGpq et A}.

Assume that $k\geq 1$ is such that the map 
$$
H^1(N_0,Q_{k}^{G_{p,q}}) \longrightarrow H^1(N_0,Q_{k})
$$
is an isomorphism.
Then, we consider the diagram 
\footnotesize
$$
\begin{tikzcd}[column sep = small, row sep = small]
    H^0(N_0,Q_{k}^{G_{p,q}}) \arrow[r] \arrow[d,"i_1"] & H^1(N_0,A_{k+1}^{G_{p,q}}) \arrow[r] \arrow[d,"i_2"] & H^1(N_0,Q_{k+1}^{G_{p,q}}) \arrow[r] \arrow[d,"i_3"] &  H^1(N_0,Q_{k}^{G_{p,q}}) \arrow[r] \arrow[d,"i_4"] & H^2(N_0,A_{k+1}^{G_{p,q}}) \arrow[d,"i_5"] \\
    H^0(N_0,Q_{k}) \arrow[r] & H^1(N_0,A_{k+1}) \arrow[r] & H^1(N_0,Q_{k+1}) \arrow[r] &  H^1(N_0,Q_{k}) \arrow[r] & H^2(N_0,A_{k+1}).
\end{tikzcd}
$$
\normalsize
By assumption, the map $i_4$ is an isomorphism. According to Lemma \ref{Lemma iso section globales LAMBDA Gpq et LAMBDA}, $i_1$ is an isomorphism. According to Proposition \ref{Proposition isomorphism cohomologie AGpq et A}, the maps $i_2$ and $i_5$ are isomorphisms. Therefore, applying the Five lemma, we deduce that the map $i_3$ is an isomorphism, which is what we wanted. Hence, we prove the following theorem.

\begin{Theorem}\label{Theorem Kuranishi family of a LVM manifold of type 264}
    The family $\varpi_{p,q}$, or the family $\varpi_{p}$, depending on the resonant case, is the Kuranishi family of the compact complex manifold $N_0$.
\end{Theorem}

\begin{proof}\label{Preuve Theorem Kuranishi family of a LVM manifold of type 264}
    We will show that the family $\varpi_{p,q}$ (or $\varpi_p$) is complete. Let $\gamma$ be a germ of family of deformation over $(\mathbb C,0)$. We see $\gamma$  as an element of the set $H^1(N_0,\Lambda)$, according to Proposition \ref{Proposition Interet du faisceau Lambda}. For every $k\geq 1$, we consider $\gamma_k$ the projection of $\gamma$ on $H^1(N_0,Q_k)$. The element $\gamma_k$ can be seen as the jet of order $k$ of the germ of family $\gamma$.

    Since the map 
    $$
    H^1(N_0,Q_{k}^{G_{p,q}}) \longrightarrow H^1(N_0,Q_{k})
    $$
    is an isomorphism, there exists a unique element $\gamma_k^{G_{p,q}}$ of $H^1(N_0,Q_{k}^{G_{p,q}})$ whose image is $\gamma_k$.

    Moreover, the family $(\gamma_k^{G_{p,q}})_k$ is compatible in the sense that the projection of $\gamma_{k+1}^{G_{p,q}}$ on $H^1(N_0,Q_{k}^{G_{p,q}})$ is $\gamma_k^{G_{p,q}}$, since $(\gamma_k)_k$ is compatible. 

    Therefore, the family $(\gamma_k^{G_{p,q}})_k$ produces a formal germ of family of deformations of the $(G_{p,q},V)$-structure of the initial one on $N_0$, such that the induced germ of family of deformations of the complex structure of $N_0$ is $\gamma$. By the Artin theorem (see \cite{Art68} Theorem 1.2), there exists a germ of family of deformations of the $(G_{p,q},V)$-structure having the same property. 
    This germ can be seen as a germ of curve in $\Hom(\mathbb Z^3,G_{p,q})$, as we have seen before. But according to Lemma \ref{Lemma map between the representations varieties for Gpq}, this germ of curve can be projected on $\Hom(\mathbb Z^2,G_{p,q})$ via the constructed map. This shows that the germ of family of deformations $\gamma$ is obtained by pulling back the family $\varpi_{p,q}$ (or $\varpi_p$) via this germ of curve. 

    More generally, as in \cite{Ghys95} (p.126), if we take any germ of complex space $(S,s)$ instead of $(\mathbb C,0)$, the same proof shows that any germ of family of deformations of $N_0$ parameterized by $(S,s)$ is obtained by pulling back the family $\varpi_{p,q}$ (or $\varpi_p$). Therefore, the family $\varpi_{p,q}$ (or $\varpi_p$) is complete.

    To conclude the proof, we use Lemma \ref{Lemma Zariski tangent space of Tpq and Sp} \textit{(ii)}. Since the Kodaira-Spencer map of the family $\varpi$ is an isomorphism and since the family is complete, $\varpi$ has the versal property, so it is the Kuranishi family of the compact complex manifold $N_0$.
\end{proof}

\subsection{Geometrization of LVM manifolds of type (2,6,4)}

As mentioned in Remark \ref{Remark On the fact that a G,X manifold has an induced complex structure}, we emphasize that a $(G,X)$-structure on a smooth manifold with $G$ a complex Lie group acting holomorphically on $X$ a complex manifold induces a complex structure of $M$, and that there exists a forgetful map 
$$
\{(G,X)-\text{structures on }M\}\longrightarrow \{\text{complex structures on }M\}.
$$
In this section, we will describe this map for an LVM manifold of type $(2,6,4)$, and for resonant structures, i.e. where $G$ is the Lie group introduced in Section \ref{Subsection Group of resonant transformations} and $X=V=\C^*\times \C^2\setminus\{0\}$.

Let $\Lambda=(\Lambda_1,...,\Lambda_6)$ be an LVM configuration of type $(2,6,4)$. We consider $N$ the associated LVM manifold. 

In the case where $N$ does not carry any non trivial resonances in the sense of Remark \ref{Remark Definition of Non resonant and affinely resonant}, the group $G$ identifies with the group $\mathrm{Diag}(\C^3)^\times$, and according to Theorem \ref{Theorem Developping map for the nearby affine structure in the non resonant case}, there exist a neighborhood $U$ of the natural homomorphism in $\Hom(\Z^3,G)$ and a holomorphic submersion $\Psi$ from $U$ to the representation space $\Hom(\Z^2,G)$ such that the image of the holonomy homomorphism of an $(G,V)$-structure gives the data of the underlying complex structure, i.e. the action of $\Z^2$ on $V$ producing an LVM manifold. This theorem also shows that, locally in the deformation space of $(G,V)$-structures, a structure is uniformizable if and only if $V$ is a holonomy covering, i.e. if the fundamental group of $V$ belongs to (and actually, is) the kernel of the holonomy homomorphism. The latter is also equivalent to the fact that the $(G,V)$-structure is the canonical one of the corresponding LVM manifold. In that case, the structure is moreover complete.

Notice that, according to Theorem \ref{Theorem Kuranishi family for LVM manifolds of type (2,6,4) non resonant}, the image of $\Psi$ is the Kuranishi space of the LVM manifold $N$. Recall that we have $N^{\mathrm{diff}}=\mathbb S^1\times\mathbb S^1 \times \mathbb S^1\times \mathbb S^3$.

\begin{Theorem}\label{Theorem the map between affine structures and complex structure in non resonant case}
    The space $U$ is a smooth open subset of the character variety of $(G,V)$-structures (see Section \ref{Subsection G,X structures}) around the natural one and the map $\Psi:U\to \Psi(U)$ is the restriction to $U$ of the map 
    $$
    \{(G,X)-\text{structures on } N^{\mathrm{diff}}\}\longrightarrow \{\text{complex structures on }N^{\mathrm{diff}}\}.
    $$
    Moreover, the image $\Psi(U)$ is an open set of the Kuranishi space of all its points, and $\Psi$ is a holomorphic submersion.
\end{Theorem}

One of the consequences of this Theorem is that, locally around the complex structure of the LVM manifold, every complex structure is geometric, in the sense that it is induced by a $(G,V)$-structure. 

Assume now that $N$ carries resonances given by a couple $(p,q)\in \Z\times\N^*\setminus \{(0,1)\}$ (see Section \ref{Lemma Classification resonances 264}). According to Lemma \ref{Lemma map between the representations varieties for Gpq}, there exists a neighborhood $U$ of the natural homomorphism in $\Hom(\Z^3,G_{p,q})$ and an analytic fibration with smooth fibers $\Psi$ from $U$ to the representation space $\Hom(\Z^2,G_{p,q})$ such that the image of a $(G_{p,q},V)$-structure gives the parameters of its complex underlying structure. The proof of this lemma also shows that, locally in the deformation space of $(G_{p,q},V)$-structures, a structure is complete if and only if it is uniformizable if and only if $V$ is a holonomy covering, i.e. if and only if the $(G_{p,q},V)$-structure is the natural one of the corresponding compact complex manifold $V/\Z^2$ (that is not, \textit{a priori}, an LVM manifold). 

Notice that, according to Theorem \ref{Theorem Kuranishi family of a LVM manifold of type 264}, the image of $\Psi$ is the Kuranishi space of the LVM manifold $N$. 

\begin{Theorem}\label{Theorem the map between resonant structures and complex structure in the resonant case}
    The space $U$ is an open subset of the representations variety associated with the $(G_{p,q},V)$-structures around the natural one and the analytic fibration with smooth fibers $\Psi:U\to \Psi(U)$ is the restriction to $U$ of the map 
    $$
    \{(G_{p,q},V)-\text{structures on }N^{\mathrm{diff}}\}\longrightarrow \{\text{complex structures on }N^{\mathrm{diff}}\}.
    $$
    Moreover, the image $\Psi(U)$ is the Kuranishi space of $N$.
\end{Theorem}

Putting together Theorem \ref{Theorem the map between affine structures and complex structure in non resonant case} and Theorem \ref{Theorem the map between resonant structures and complex structure in the resonant case}, we infer that every small deformation of an LVM manifold of type $(2,6,4)$ is geometric, and the group of symmetries is given by the resonances of the LVM manifold acting on $V=\C^*\times \C^2\setminus\{0\}$.

In conclusion, we summarize our results in the following statement.

\begin{Theorem}\label{Theorem Geometrization of LVM manifolds}
    Every small deformation of an LVM manifold $N$ of type $(2,6,4)$ carries a family of compatible, i.e. holomorphic, resonant structures (for $N$) which are parametrized by $3,4$ or $5$ complex numbers. Moreover, among these structures, there exists a unique one that is uniformizable, and it has the property of being complete.
\end{Theorem}

\section{A complete family of LVM manifolds}\label{Subsection A complete family of LVM manifolds}

Dabrowski proposed in \cite{Dab82} a way of gluing families of Hopf surfaces in order to build a family of deformations that contains all isomorphism classes of Hopf surfaces. The first remark is that this is only possible for surfaces. Indeed, the class of Hopf manifold of dimension $n\geq 3$ has the property that the dimension of the Kuranishi space is unbounded. It is equivalent to the fact that for a Hopf manifold of dimension $n\geq 3$ the number of resonances can be arbitrarily large. The family of Dabrowski contains all the Hopf surfaces, however, it does not have the property of being complete. A family $\varpi$ of deformations of $M$ is complete if every other family of deformations of $M$ is obtained by a pullback of $\varpi$. 

In his thesis \cite{Fro2017}, Fromenteau revisits the work of Dabrowski and manages to construct a family of Hopf surfaces that contains all of them and that is complete at every point. Fromenteau furthermore uses his construction to describe the moduli space of complex structures on $\mathbb S^1\times \mathbb S^3$ from the Teichmüller analytic stacks point of view.

LVM manifolds of type $(2,6,4)$ have, like Hopf surfaces, the property of satisfying a bounded number of resonance relations. Therefore, it is natural to hope that we can glue the Kuranishi families of LVM manifolds in order to get a family that contains all LVM manifolds of type $(2,6,4)$ and that is complete at every point. We draw inspiration from the construction of Fromenteau in order to build such a family.

We use the notation of the proof of Lemma \ref{Lemma common normal form for commuting resonant transfo}. For $p\in\Z$, and $(\alpha,A)\in\C^*\times \mathrm{GL}(2,\C)$, we say that $\lambda$ is a \textbf{$p$-eigenvalue} if the space $\ker(A-\lambda L_{\alpha,p})$ is non-trivial. We recall that $L_{\alpha,p}$ is defined by 
$$
L_{\alpha,p}=\begin{pmatrix}
    1 & 0 \\ 0 & \alpha^p
\end{pmatrix}.
$$
We also say that $\alpha$ is a $p$-eigenvalue for any $p\in\Z$.

\begin{Definition}\label{Definition pour construire l'ensemble Sp}
    Let $p$ be an integer and let 
    $$
    \left(
\begin{pmatrix}
    \alpha_1 & 0 & 0 \\
    0 & \alpha_2 & \varepsilon_2 \\
    0 & \varepsilon_1 & \alpha_3
\end{pmatrix},
\begin{pmatrix}
    \beta_1 & 0 & 0 \\
    0 & \beta_2 & \delta_2 \\
    0 & \delta_1 & \beta_3
\end{pmatrix}
\right)
    $$
    be an element of $\mathrm{GL}(3,\C)^2$. We say that the latter verifies the condition $C_{p}$ if 
    \begin{enumerate}
        \item $\left\{
        \begin{matrix}
            \varepsilon_1\delta_2\beta_1^{p}&=&\delta_1\varepsilon_2\alpha_1^{p}  \\
    \varepsilon_1(\beta_3-\beta_1^p\beta_2)&=&\delta_1(\alpha_3-\alpha_1^p\alpha_2)  \\
    \varepsilon_2(\beta_2-\beta_1^{-p}\beta_3)&=&\delta_2(\alpha_2-\alpha_1^{-p}\alpha_3) 
        \end{matrix}
        \right.$,
        \item the $p$-eigenvalues of the two matrices come from an LVM configuration of type $(2,6,4)$,
        \item for every $(r,s)\in \Z\times \N^*$ distinct of $(p,1)$, the $p$-eigenvalues $\alpha_1,\alpha_2'$ and $\alpha_3'$ of the first matrix verify $\alpha_3'\neq\alpha_1^r\alpha_2'^s$.
    \end{enumerate}

    We say that it verifies the condition $C_p^{S}$ if it verifies the condition $C_p$ and if the $p$-eigenvalues $\alpha_1,\alpha_2'$ and $\alpha_3'$ of the first matrix and the $p$-eigenvalues $\beta_1,\beta_2'$ and $\beta_3'$ verify $$
    \left \{
    \begin{matrix}
        \alpha_3' &= &\alpha_1^p\alpha_2' \\
        \beta_3' &= &\beta_1^p\beta_2'
    \end{matrix}
    \right..
    $$
\end{Definition}

For $p\in\Z$, we consider $S_p$ the algebraic variety defined as 
$$
S_p:=\left\{
\left(
\begin{pmatrix}
    \alpha_1 & 0 & 0 \\
    0 & \alpha_2 & \varepsilon_2 \\
    0 & \varepsilon_1 & \alpha_3
\end{pmatrix},
\begin{pmatrix}
    \beta_1 & 0 & 0 \\
    0 & \beta_2 & \delta_2 \\
    0 & \delta_1 & \beta_3
\end{pmatrix}
\right) \in \mathrm{GL}(3,\mathbb C)^2 \text{ that verifies } C_p
\right\}
$$
The variety $S_p$ has dimension $8$ and its singular locus is 
$$
\left\{
\left(
\begin{pmatrix}
    \alpha_1 & 0 & 0 \\
    0 & \alpha_2 & \varepsilon_2 \\
    0 & \varepsilon_1 & \alpha_3
\end{pmatrix},
\begin{pmatrix}
    \beta_1 & 0 & 0 \\
    0 & \beta_2 & \delta_2 \\
    0 & \delta_1 & \beta_3
\end{pmatrix}
\right) \in \mathrm{GL}(3,\mathbb C)^2 \text{ that verifies } C_p^{S}
\right\}.
$$
According to Theorem \ref{Theorem The action is fpf and proper} and Proposition \ref{Proposition Groupe des transfo résonantes} \textit{(iii)}, for every element $(A,B)$ in $S_p$ with 
$$
A=
\begin{pmatrix}
    \alpha_1 & 0 & 0 \\
    0 & \alpha_2 & \varepsilon_2 \\
    0 & \varepsilon_1 & \alpha_3
\end{pmatrix}~B=
\begin{pmatrix}
    \beta_1 & 0 & 0 \\
    0 & \beta_2 & \delta_2 \\
    0 & \delta_1 & \beta_3
\end{pmatrix},
$$
the action $\mathcal A^{p}_{A,B}:\Z^2\times V \to V$ of $\Z^2$ on $V$ given by 
$$
\begin{matrix}
    (1,0).(\xi_1,\xi_2,\xi_3)&=&(\alpha_1\xi_1,\alpha_2\xi_2+\varepsilon_2\xi_1^{-p}\xi_3,\alpha_3\xi_3+\varepsilon_1\xi_1^p\xi_2), \\
    (0,1).(\xi_1,\xi_2,\xi_3)&=&(\beta_1\xi_1,\beta_2\xi_2+\delta_2\xi_1^{-p}\xi_3,\beta_3\xi_3+\delta_1\xi_1^p\xi_2)
\end{matrix}
$$
is fixed point free, proper and cocompact. We consider the action of $\Z^2$
$$
\begin{matrix}
    \Z^2\times S_p\times V & \longrightarrow & S_p\times V \\
    ((r,s),(A,B),\xi) & \longmapsto & ((A,B),\mathcal A^{p}_{A,B}((r,s),\xi)).
\end{matrix}
$$
Obviously, this action is fixed point free and proper. Therefore, the quotient space $\mathcal M_{p}:=(S_{p}\times V)/\mathbb Z^2$ carries a natural structure of algebraic variety, and the natural projection $S_{p}\times V \to S_{p}$ descends to the quotient to a flat and proper algebraic map $\varpi_{p} :\mathcal M_{p} \to S_{p}$. It is proved in Theorem \ref{Theorem Kuranishi family for LVM manifolds of type (2,6,4) non resonant} and Theorem \ref{Theorem Kuranishi family of a LVM manifold of type 264} that, for every $p$ in $\Z$, the family $\varpi_p$ is complete at each point of $S_p$ that parametrizes an LVM manifold.

\begin{Definition}\label{Definition pour construire l'ensemble Tpq}
    Let $(p,q)$ be an element of $\Z\times \N_{\geq2}$ and let 
    $$
    \left(
\begin{pmatrix}
    \alpha_1 & 0 & 0 \\
    0 & \alpha_2 & 0 \\
    0 & \varepsilon & \alpha_3
\end{pmatrix},
\begin{pmatrix}
    \beta_1 & 0 & 0 \\
    0 & \beta_2 & 0 \\
    0 & \delta & \beta_3
\end{pmatrix}
\right)
    $$
    be an element of $\mathrm{GL}(3,\C)^2$. We say that the latter verifies the condition $K_{p,q}$ if 
    \begin{enumerate}
        \item $\vert\alpha_2\vert>\vert\alpha_3\vert$,
        \item $\varepsilon(\beta_3-\beta_1^p\beta_2^q)=\delta(\alpha_3-\alpha_1^p\alpha_2^q)$,
        \item the eigenvalues of the two matrices come from an LVM configuration of type $(2,6,4)$,
        \item for every $(r,s)\in \Z\times \N^*$ distinct of $(p,q)$, the eigenvalues $\alpha_1,\alpha_2$ and $\alpha_3$ of the first matrix verify $\alpha_3\neq\alpha_1^r\alpha_2^s$.
    \end{enumerate}

    We say that it verifies the condition $K_{p,q}^{S}$ if it verifies the condition $K_{p,q}$ and if the eigenvalues verify $$
    \left \{
    \begin{matrix}
        \alpha_3 &= &\alpha_1^p\alpha_2^q \\
        \beta_3 &= &\beta_1^p\beta_2^q
    \end{matrix}
    \right..
    $$
\end{Definition}

For $(p,q)\in \Z\times \N_{\geq2}$, we consider $T_{p,q}$ the algebraic variety defined as 
$$
T_{p,q}=T_{p,q}'\times \C
$$
where
\small
$$ 
T_{p,q}':=
\left\{
\left(
\begin{pmatrix}
    \alpha_1 & 0 & 0 \\
    0 & \alpha_2 & 0 \\
    0 & \varepsilon & \alpha_3
\end{pmatrix},
\begin{pmatrix}
    \beta_1 & 0 & 0 \\
    0 & \beta_2 & 0 \\
    0 & \delta & \beta_3
\end{pmatrix}
\right)\in \mathrm{GL}(3,\C)^2 \text{ that verifies } K_{p,q}
\right\}.
$$
\normalsize
The variety $T_{p,q}$ has dimension $8$ and its singular locus is 
\small
$$
\left\{
\left(
\begin{pmatrix}
    \alpha_1 & 0 & 0 \\
    0 & \alpha_2 & 0 \\
    0 & \varepsilon & \alpha_3
\end{pmatrix},
\begin{pmatrix}
    \beta_1 & 0 & 0 \\
    0 & \beta_2 & 0 \\
    0 & \delta & \beta_3
\end{pmatrix}
\right)\in \mathrm{GL}(3,\C)^2 \text{ that verifies } K_{p,q}^S
\right\}\times \C.
$$
\normalsize
According to Theorem \ref{Theorem The action is fpf and proper}, for every element $(A,B)$ in $T_{p,q}$ with 
$$
A=
\begin{pmatrix}
    \alpha_1 & 0 & 0 \\
    0 & \alpha_2 & 0 \\
    0 & \varepsilon & \alpha_3
\end{pmatrix}~B=
\begin{pmatrix}
    \beta_1 & 0 & 0 \\
    0 & \beta_2 & 0 \\
    0 & \delta & \beta_3
\end{pmatrix},
$$
the action $\mathcal A^{p,q}_{A,B}:\Z^2\times V \to V$ of $\Z^2$ on $V$ given by 
$$
\begin{matrix}
    (1,0).(\xi_1,\xi_2,\xi_3)=(\alpha_1\xi_1,\alpha_2\xi_2,\alpha_3\xi_3+\varepsilon\xi_1^p\xi_2^q), \\
    (0,1).(\xi_1,\xi_2,\xi_3)=(\beta_1\xi_1,\beta_2\xi_2,\beta_3\xi_3+\delta\xi_1^p\xi_2^q)
\end{matrix}
$$
is fixed point free, proper and cocompact. We consider the action of $\Z^2$
$$
\begin{matrix}
    \Z^2\times T_{p,q}\times V & \longrightarrow & T_{p,q}\times V \\
    ((r,s),(A,B),\lambda,\xi) & \longmapsto & ((A,B),\lambda,\mathcal A^{p,q}_{A,B}((r,s),\xi)).
\end{matrix}
$$
Obviously, this action is fixed point free and proper. Therefore, the quotient space $\mathcal M_{p,q}:=(T_{p,q}\times V)/\mathbb Z^2$ carries a natural structure of algebraic variety, and the natural projection $T_{p,q}\times V \to T_{p,q}$ descends to the quotient to a flat and proper algebraic map $\varpi_{p,q} :\mathcal M_{p,q} \to T_{p,q}$. It is proved in Theorem \ref{Theorem Kuranishi family for LVM manifolds of type (2,6,4) non resonant} and Theorem \ref{Theorem Kuranishi family of a LVM manifold of type 264} that, for every $(p,q)$ in $\Z\times \N_{\geq2}$, the family $\varpi_{p,q}$ is complete at each point of $T_{p,q}$ that parametrizes an LVM manifold.

\begin{Definition}\label{Definition pour construire l'ensemble T}
    Let 
    $$
    \left(
\begin{pmatrix}
    \alpha_1 & 0 & 0 \\
    0 & \alpha_2 & 0 \\
    0 & \varepsilon & \alpha_3
\end{pmatrix},
\begin{pmatrix}
    \beta_1 & 0 & 0 \\
    0 & \beta_2 & 0 \\
    0 & \delta & \beta_3
\end{pmatrix}
\right)
    $$
    be an element of $\mathrm{GL}(3,\C)^2$. We say that the latter verifies the condition $C$ if 
    \begin{enumerate}
        \item $\vert\alpha_2\vert>\vert\alpha_3\vert$,
        \item $\varepsilon(\beta_3-\beta_2)=\delta(\alpha_3-\alpha_2)$,
        \item the eigenvalues of the two matrices come from an LVM configuration of type $(2,6,4)$,
        \item for every $(r,s)\in \Z\times \N^*$, the eigenvalues $\alpha_1,\alpha_2$ and $\alpha_3$ of the first matrix verify $\alpha_3\neq\alpha_1^r\alpha_2^s$.
    \end{enumerate}
\end{Definition}

We consider $T$ the complex manifold defined as 
$$
T=T'\times \C
$$
where
\small
$$ 
T':=
\left\{
\left(
\begin{pmatrix}
    \alpha_1 & 0 & 0 \\
    0 & \alpha_2 & 0 \\
    0 & \varepsilon & \alpha_3
\end{pmatrix},
\begin{pmatrix}
    \beta_1 & 0 & 0 \\
    0 & \beta_2 & 0 \\
    0 & \delta & \beta_3
\end{pmatrix}
\right)\in \mathrm{GL}(3,\C)^2 \text{ that verifies } C
\right\}.
$$
\normalsize
The manifold $T$ has dimension $8$.

According to Theorem \ref{Theorem The action is fpf and proper}, for every element $(A,B)$ in $T$ with 
$$
A=
\begin{pmatrix}
    \alpha_1 & 0 & 0 \\
    0 & \alpha_2 & 0 \\
    0 & \varepsilon & \alpha_3
\end{pmatrix}~B=
\begin{pmatrix}
    \beta_1 & 0 & 0 \\
    0 & \beta_2 & 0 \\
    0 & \delta & \beta_3
\end{pmatrix},
$$
the action $\mathcal A_{A,B}:\Z^2\times V \to V$ of $\Z^2$ on $V$ given by 
$$
\begin{matrix}
    (r,s).\xi=A^rB^s\xi
\end{matrix}
$$
is fixed point free, proper and cocompact. We consider the action of $\Z^2$
$$
\begin{matrix}
    \Z^2\times T\times V & \longrightarrow & T\times V \\
    ((r,s),(A,B),\lambda,\xi) & \longmapsto & ((A,B),\lambda,A^rB^s\xi)).
\end{matrix}
$$
Obviously, this action is fixed point free and proper. Therefore, the quotient space $\mathcal M:=(T\times V)/\mathbb Z^2$ carries a natural structure of complex manifold, and the natural projection $T\times V \to T$ descends to the quotient to a holomorphic submersion $\varpi :\mathcal M \to T$. It is proved in Theorem \ref{Theorem Kuranishi family for LVM manifolds of type (2,6,4) non resonant} that the family $\varpi$ is complete at each point of $T$ that parametrizes an LVM manifold.

We will glue all the family we constructed together. As in \cite{Fro2017} \S 7.2.1 p.59, the family $\varpi$ will only serve to define the gluing maps. 

For $p\in Z$, we consider the map 
$$
\begin{matrix}
    \psi_p & : & T\times V & \longrightarrow & S_p\times V 
\end{matrix}
$$
that sends an element 
$$
\left(\left(
\begin{pmatrix}
    \alpha_1 & 0 & 0 \\
    0 & \alpha_2 & 0 \\
    0 & \varepsilon & \alpha_3
\end{pmatrix},
\begin{pmatrix}
    \beta_1 & 0 & 0 \\
    0 & \beta_2 & 0 \\
    0 & \delta & \beta_3
\end{pmatrix}
\right),\lambda,(\xi_1,\xi_2,\xi_3)\right)
$$
to the element 
$$
\left(
\begin{matrix}
\begin{pmatrix}
    1 & 0 & 0 \\
    0 & 1 & \lambda\alpha_1^{-p}\\
    0 & 0 & 1
\end{pmatrix}
\begin{pmatrix}
    \alpha_1 & 0 & 0 \\
    0 & \alpha_2 & 0 \\
    0 & \varepsilon & \alpha_3
\end{pmatrix}
\begin{pmatrix}
    1 & 0 & 0 \\
    0 & 1 & -\lambda\\
    0 & 0 & 1
\end{pmatrix}, \\ \\
\begin{pmatrix}
    1 & 0 & 0 \\
    0 & 1 & \lambda\beta_1^{-p}\\
    0 & 0 & 1
\end{pmatrix}
\begin{pmatrix}
    \beta_1 & 0 & 0 \\
    0 & \beta_2 & 0 \\
    0 & \delta_1 & \beta_3
\end{pmatrix}
\begin{pmatrix}
    1 & 0 & 0 \\
    0 & 1 & -\lambda\\
    0 & 0 & 1
\end{pmatrix}, \\ \\
\begin{pmatrix}
    \xi_1 \\ \\
    \begin{pmatrix}
        1 & \lambda\xi_1^{-p}\\
        0 & 1
    \end{pmatrix}
    \begin{pmatrix}
        \xi_2 \\
        \xi_3+\frac{\varepsilon}{\alpha_3-\alpha_2}\xi_2-\frac{\varepsilon}{\alpha_3-\alpha_1^p\alpha_2}\xi_1^p\xi_2
    \end{pmatrix}
\end{pmatrix}
\end{matrix}\right),
$$
where $\delta_1=\varepsilon\frac{\beta_3-\beta_1^p\beta_2}{\alpha_3-\alpha_1^p\alpha_2}$.

\begin{Lemma}\label{Lemma Gluing map for Sp}
    For every $p$ in $\Z$, the map $\psi_p$ descends to a morphism of families of deformations between $\varpi$ and $\varpi_p$, i.e. is equivariant for the corresponding actions of $\Z^2$. Moreover, the map $\psi_p$ is one-to-one and its image is the smooth open set of $S_p\times V$ defined by $W_p\times V$, where 
    $$
    W_p:=\left\{
    \begin{matrix}
    \left(
\begin{pmatrix}
    \alpha_1 & 0 & 0 \\
    0 & \alpha_2 & \varepsilon_2 \\
    0 & \varepsilon_1 & \alpha_3
\end{pmatrix},
\begin{pmatrix}
    \beta_1 & 0 & 0 \\
    0 & \beta_2 & \delta_2 \\
    0 & \delta_1 & \beta_3
\end{pmatrix}
\right)\in S_p \\ ~\text{the } p\text{-eigenvalues of the first matrix verify } \left\{\begin{matrix}
    \alpha_3'\neq\alpha_1^p\alpha_2' \\
    \vert\alpha_2'\vert>\vert\alpha_3'\vert\\
    \alpha_2=\alpha_2'\text{ if }\varepsilon_1=0
\end{matrix}\right.
\end{matrix}
\right\}.
    $$
\end{Lemma}

\begin{proof}
    The equivariance property follows from direct computations of the quantities 
    $$
    \begin{matrix}
        \psi_p(A,B,\lambda,A^rB^s\xi) & \text{and} & (r,s).\psi_p(A,B,\lambda,\xi),
    \end{matrix}
    $$
    where the last action is given by $\mathcal A^p$. 

    Assume that 
    $$
\left(\left(
\begin{pmatrix}
    \alpha_1 & 0 & 0 \\
    0 & \alpha_2 & 0 \\
    0 & \varepsilon & \alpha_3
\end{pmatrix},
\begin{pmatrix}
    \beta_1 & 0 & 0 \\
    0 & \beta_2 & 0 \\
    0 & \delta & \beta_3
\end{pmatrix}
\right),\lambda,(\xi_1,\xi_2,\xi_3)\right)
$$ and 
$$
\left(\left(
\begin{pmatrix}
    \alpha_1' & 0 & 0 \\
    0 & \alpha_2' & 0 \\
    0 & \varepsilon' & \alpha_3'
\end{pmatrix},
\begin{pmatrix}
    \beta_1' & 0 & 0 \\
    0 & \beta_2' & 0 \\
    0 & \delta' & \beta_3'
\end{pmatrix}
\right),\lambda',(\xi_1',\xi_2',\xi_3')\right)
$$
are two elements of $T\times V$ that have the same image through the map $\psi_p$. It is clear that it implies $\alpha_1=\alpha_1'$, $\beta_1=\beta_1'$, $\varepsilon=\varepsilon'$ and $\xi_1=\xi_1'$. Now, the complex numbers $\alpha_2$ and $\alpha_3$, as well as $\alpha_2'$ and $\alpha_3'$, are the $p$-eigenvalues of the first matrix of the image in the space $\Vect(e_2,e_3)$, where $(e_1,e_2,e_3)$ is the standard basis of $\C^3$. By the condition $C$ of Definition \ref{Definition pour construire l'ensemble T}, we have $\vert\alpha_2\vert>\vert\alpha_3\vert$ and $\vert\alpha_2'\vert>\vert\alpha_3'\vert$, therefore, we have $\alpha_2=\alpha_2'$ and $\alpha_3=\alpha_3'$. This implies $\lambda=\lambda'$, which itself implies $\beta_2=\beta_2'$, $\beta_3=\beta_3'$, $\delta=\delta'$, $\xi_2=\xi_2'$ and finally $\xi_3=\xi_3'$. Hence, the map $\psi_p$ is one-to-one. 

It is clear that the direct image of $\psi_p$ is included in $W_p\times V$. Assume that 
$$
\left(
\begin{pmatrix}
    \alpha_1' & 0 & 0 \\
    0 & \alpha_2' & \varepsilon_2' \\
    0 & \varepsilon_1' & \alpha_3'
\end{pmatrix},
\begin{pmatrix}
    \beta_1' & 0 & 0 \\
    0 & \beta_2' & \delta_2' \\
    0 & \delta_1' & \beta_3'
\end{pmatrix},(\xi_1',\xi_2',\xi_3')
\right)
$$
is an element of $W_p\times V$. Then we consider $\alpha_1=\alpha_1'$, $\beta_1=\beta_1'$, $\varepsilon=\varepsilon_1'$ and $\xi_1=\xi_1'$. By the definition of $W_p$, the first matrix has two $p$-eigenvalues $\alpha_2$ and $\alpha_3$ in $\Vect(e_2,e_3)$ that satisfy $\vert\alpha_2\vert>\vert\alpha_3\vert$. Since $\vert\alpha_2\vert>\vert\alpha_3\vert$ and $\alpha_2'=\alpha_2$ if $\varepsilon_1'=0$, there exists a unique $\lambda\in\C$ such that $(0,\lambda,1)$ is a $p$-eigenvector of the first matrix for the $p$-eigenvalue $\alpha_3$. Notice that, by the commutation assumptions, $(0,\lambda,1)$ is also a $p$-eigenvector of the second matrix. This is enough to define properly $\beta_2$ and $\beta_3$, and therefore $\delta$. Knowing $\lambda$ and $\xi_1$ also permits to define $\xi_2$ and then $\xi_3$ such that 
$$
\left(\left(
\begin{pmatrix}
    \alpha_1 & 0 & 0 \\
    0 & \alpha_2 & 0 \\
    0 & \varepsilon & \alpha_3
\end{pmatrix},
\begin{pmatrix}
    \beta_1 & 0 & 0 \\
    0 & \beta_2 & 0 \\
    0 & \delta & \beta_3
\end{pmatrix}
\right),\lambda,(\xi_1,\xi_2,\xi_3)\right)
$$ is mapped by $\psi_p$ to 
$$
\left(
\begin{pmatrix}
    \alpha_1' & 0 & 0 \\
    0 & \alpha_2' & \varepsilon_2' \\
    0 & \varepsilon_1' & \alpha_3'
\end{pmatrix},
\begin{pmatrix}
    \beta_1' & 0 & 0 \\
    0 & \beta_2' & \delta_2' \\
    0 & \delta_1' & \beta_3'
\end{pmatrix},(\xi_1',\xi_2',\xi_3')
\right)
$$
This shows that $W_p\times V$ is the direct image of $\psi_p$.
\end{proof}

For $(p,q)\in\Z\times \N_{\geq2}$, we define the smooth open subset of $T_{p,q}$ by 
$$
U_{p,q}=\left\{\left(
\begin{pmatrix}
    \alpha_1 & 0 & 0 \\
    0 & \alpha_2 & 0 \\
    0 & \varepsilon & \alpha_3
\end{pmatrix},
\begin{pmatrix}
    \beta_1 & 0 & 0 \\
    0 & \beta_2 & 0 \\
    0 & \delta & \beta_3
\end{pmatrix},\lambda
\right)\in T_{p,q},~\alpha_3\neq\alpha_1^p\alpha_2^q
\right\}
$$
and the map 
$$
\begin{matrix}
    \phi_{p,q} & : & U_{p,q}\times V & \longrightarrow & T\times V 
\end{matrix}
$$
that sends an element 
$$
\left(\left(
\begin{pmatrix}
    \alpha_1 & 0 & 0 \\
    0 & \alpha_2 & 0 \\
    0 & \varepsilon & \alpha_3
\end{pmatrix},
\begin{pmatrix}
    \beta_1 & 0 & 0 \\
    0 & \beta_2 & 0 \\
    0 & \delta & \beta_3
\end{pmatrix}
\right),\lambda,(\xi_1,\xi_2,\xi_3)\right)
$$
to the element 
$$
\left(\left(
\begin{pmatrix}
    \alpha_1 & 0 & 0 \\
    0 & \alpha_2 & 0 \\
    0 & \varepsilon & \alpha_3
\end{pmatrix},
\begin{pmatrix}
    \beta_1 & 0 & 0 \\
    0 & \beta_2 & 0 \\
    0 & \delta' & \beta_3
\end{pmatrix}\right),\lambda
\begin{pmatrix}
    \xi_1 \\
    
        \xi_2 \\
        \xi_3+\frac{\varepsilon}{\alpha_3-\alpha_2}\xi_2-\frac{\varepsilon}{\alpha_3-\alpha_1^p\alpha_2}\xi_1^p\xi_2
    
\end{pmatrix}
\right),
$$
where $\delta_1=\varepsilon\frac{\beta_3-\beta_2}{\alpha_3-\alpha_2}$.

\begin{Lemma}\label{Lemma Gluing map for Tpq}
    For every $(p,q)$ in $\Z\times\N_{\geq2}$, the map $\phi_{p,q}$ descends to a morphism of families of deformations between $\varpi_{p,q}$ and $\varpi$, i.e. is equivariant for the corresponding actions of $\Z^2$. Moreover, the map $\phi_{p,q}$ is a biholomorphism.
\end{Lemma}

\begin{proof}
    The equivariance property follows from direct computations of the quantities 
    $$
    \begin{matrix}
        \psi_p(A,B,\lambda,\mathcal A^{p,q}_{A,B}((r,s),\xi)) & \text{and} & (r,s).\psi_p(A,B,\lambda,\xi),
    \end{matrix}
    $$
    where the last action is given by $\mathcal A$. 

    The fact that $\phi_{p,q}$ is a biholomorphism is straightforward.
\end{proof}

Using Lemma \ref{Lemma Gluing map for Sp} and Lemma \ref{Lemma Gluing map for Tpq}, we glue the families $\varpi_{p'}$ and $\varpi_{p,q}$ for every $p'\in\Z$ and $(p,q)\in\Z\times \N_{\geq2}$ all together to construct a family containing all LVM manifolds of type $(2,6,4)$ such that this family is complete at every LVM manifold. 

\begin{Theorem}\label{Theorem big family of LVM manifolds}
    There exists a family of deformations that contains all LVM manifolds of type $(2,6,4)$ and that is complete at every such a manifold.
\end{Theorem}

\newpage

\Large
\noindent\textbf{Acknowledgments}
\normalsize

We thank Sorin Dumitrescu and Laurent Meersseman for their helpful support throughout the development of this work. We also wish to express our sincere gratitude to Sönke Rollenske for his interest in the results of this paper and for the valuable discussions we had on this topic.

\nocite{*}
    \footnotesize{\bibliographystyle{alpha}\bibliography{main}}

\vspace{0.5cm}
\texttt{UNIVERSITÉ COTE D'AZUR, LJAD, FRANCE}

Email address: \href{mailto:matthieu.madera@univ-cotedazur.fr}{matthieu.madera@univ-cotedazur.fr}

\end{document}